\documentclass[10pt,reqno,a4]{amsart}

\numberwithin{equation}{section}

\usepackage[latin1]{inputenc}
\usepackage[english]{babel}

\usepackage{amsmath,amsthm,amsfonts,latexsym,amssymb}
\usepackage[colorlinks]{hyperref}
\hypersetup{linkcolor=blue,citecolor=blue,filecolor=black,urlcolor=blue}
\usepackage{comment}
\usepackage{soul}

\usepackage[
  hmarginratio={1:1},     
  vmarginratio={1:1},     
  textwidth=15cm,        
  textheight=20cm,  
  heightrounded,          
]{geometry}

{ \theoremstyle{plain}
\newtheorem{theorem}{Theorem}[section]
\newtheorem{proposition}[theorem]{Proposition}
\newtheorem{lemma}[theorem]{Lemma}

  \theoremstyle{remark}

  \theoremstyle{definition}
\newtheorem{definition}[theorem]{Definition}
\newtheorem{example}[theorem]{Example}
}

\def\R{\mathbb{R}}

\DeclareMathOperator{\Ker}{Ker}

\DeclareMathOperator{\spann}{span}
\DeclareMathOperator{\dist}{dist}

\def\eps{\varepsilon}

\def\eps{\varepsilon}

\def\deltab{\boldsymbol{\delta}}

\def\xib{\boldsymbol{\xi}}
\def\phib{\boldsymbol{\phi}}
\def\psib{\boldsymbol{\psi}}
\def\lambdab{\boldsymbol{\lambda}}

\def\ob{\textrm{o}}
\def\Ob{\textrm{O}}

\newcommand{\inner}[2]{\langle  #1,{ #2}\rangle_{H^1_0}}

\begin{document}

\title[]{Spiked Solutions for Schr\"odinger Systems with Sobolev Critical Exponent: the cases of competitive and weakly cooperative interactions}

\author[A. Pistoia]{Angela Pistoia}
\address{Angela Pistoia \newline \indent Universit\`a di Roma ``La Sapienza'' \newline \indent
Dipartimento di Metodi e Modelli Matematici,  via Antonio Scarpa 16, \newline 
\indent 00161 Roma, Italy}
\email{angela.pistoia@uniroma1.it}

\author[H. Tavares]{Hugo Tavares}
\address{Hugo Tavares \newline \indent CAMGSD, Instituto Superior T\'ecnico
\newline \indent Pavilh\~ao de Matem\'atica, Av. Rovisco Pais \newline \indent
1049-001 Lisboa, Portugal}
\email{htavares@math.ist.utl.pt}

\date{\today}


\maketitle

\begin{abstract} 
In this paper we deal with the nonlinear Schr\"odinger system
\[
-\Delta u_i =\mu_i u_i^3 + \beta u_i \sum_{j\neq i} u_j^2 + \lambda_i u_i, \qquad u_1,\ldots, u_m\in H^1_0(\Omega)
\]
in dimension 4, a problem with critical Sobolev exponent. In the competitive case ($\beta<0$ fixed or $\beta\to -\infty$) or in the weakly cooperative case ($\beta\geq 0$ small), we construct, under suitable assumptions on the Robin function associated to the domain $\Omega$, families of positive solutions which blowup and concentrate at different points as $\lambda_1,\ldots, \lambda_m\to 0$. This problem can be seen as a generalization for systems of a Brezis-Nirenberg type problem.
\end{abstract}
\section{Introduction}
Consider the cubic Schr\"odinger system with $m$ equations:
\begin{equation}\label{eq:system_2eq}
\begin{cases}
\displaystyle -\Delta u_i =\mu_i u_i^3 + \beta u_i \mathop{\sum_{j=1}^m}_{j\neq i} u_j^2 + \lambda_i u_i,\quad i=1,\ldots, m\\
u_1,\ldots, u_m \in H^1_0(\Omega)
\end{cases}
\end{equation}
where $\Omega$ is a bounded domain in $\R^4$. Such system appears when looking for standing wave solutions $\phi_i=e^{-\imath \lambda_i t}u_i(x)$ of the corresponding system of Gross-Pitaevskii equations
\[
\imath \partial_t \phi_i -\Delta \phi_i =\mu_i |\phi_i|^2 \phi_i + \beta \phi_i \sum_{j\neq i} |\phi_j|^2.
\] 
The sign and size of the parameter $\beta$ determines the type and strengh of the interaction between different components of the vector solution. If $\beta>0$, the interaction is of \emph{cooperative} nature, and the system is used to describe phenomena in Nonlinear Optics, for instance describing the propagation of self trapped mutually incoherent wave packets \cite{Fisica1}; in this situation, $\phi_i$ describes the $i$--th components of a beam. On the other hand, if $\beta<0$, then the interaction is \emph{competitive}, and the system has been used in the theory of Bose-Einstein Condensation  to model the presence of several distinguishable consensates \cite{Fisica2}; here $\phi_i$ is the wave function of the $i$-th condensate.

\medbreak

From a mathematical point of view, \eqref{eq:system_2eq} is a good prototype of a weakly coupled gradient system. We work in dimension 4, so that the exponent $3$ is critical. Hence, the problem can be viewed as a generalization for systems of the well known \emph{Brezis-Nirenberg problem}. In this paper we construct, under some geometric assumptions on the domain $\Omega$, and for either $-\infty<\beta\leq \bar \beta$ ($\bar \beta>0$ small) or as $\beta\to -\infty$, solutions which concentrate and blowup at different points as $\lambda_1,\ldots, \lambda_m\to 0$. %

\medbreak

The classical Brezis-Nirenberg problem:
\begin{equation}\label{eq:BN}
-\Delta u=u^{(N+2)/(N-2)}+ \lambda u \text{ in } \Omega\subset \R^N,\qquad u=0 \text{ on } \partial \Omega
\end{equation}
has a long history. Brezis and Nirenberg, in their seminal paper \cite{BrezisNirenberg}, have proved among other things that in dimension $N\geq 4$, problem \eqref{eq:BN} admits a positive solution if and only if $0<\lambda<\lambda_1(\Omega)$, where $\lambda_1(\Omega)$ denotes the first eigenvalue of $(-\Delta, H^1_0(\Omega))$. Han, in \cite{HanAIHP}, proved that, if $(u_\lambda)_\lambda$ is a family of positive solution of \eqref{eq:BN}, being also a minimizing sequence for the best Sobolev constant of $H^1(\R^N)\hookrightarrow L^{2N/(N-2)}(\R^N)$, then as $\lambda\to 0$ the functions concentrate at a critical point of the Robin function of $\Omega$. Conversely, if $N\geq 5$, Rey in \cite{ReyNA,Rey} proved that, given $x_0$ a non-degenerate critical point of the Robin function, then there exists a family of solutions concentrating at $x_0$, as $\lambda\to 0$. This last result was later improved by Musso and Pistoia \cite{MussoPistoiaIndiana2002}, where solutions concentrating in $k\geq 1$ different points of the domain were found. 

The study of existence and concentration of sign-changing solutions of \eqref{eq:BN} is a much more delicate problem, and we refer for instance to the papers \cite{Pacella,BartschMichelettiPistoia,CastroClapp,Giusi, MichelettiPistoiaNonlinearity2004}. The reference \cite{PistoiaSurvey} is a survey containing a rather complete description of the litterature, not only for the Brezis-Nirenberg case, but also for the \emph{almost critical problem} and the \emph{Bahri-Coron problem} for the single equation.

Let us stress that the vast majority of the papers we cited deal with the case $N\geq 5$: the case $N=4$ deserves in general an extra care from a technical point of view, mainly due to the fact that the liming profile (i.e., the solutions of $-\Delta U=U^3$ in the whole $\R^4$) do not belong to the space $L^2(\R^4)$. Even if many results can be extended for $N=4$, few papers rigorously state and prove this: we refer for instance to \cite{Pacella,Giusi}.

\medbreak

Consider now the system
\begin{equation}\label{eq:systemp}
-\Delta u_i =\mu_i |u_i|^{2p-2}u_i + \beta u_i |u_i|^{p-2} \sum_{j\neq i} |u_j|^{p} + \lambda_i u_i,\qquad i=1,\ldots, m,
\end{equation}
with $2p\leq 2N/(N-2)$, of which \eqref{eq:system_2eq} is a particular case for $p=2$. The existence of positive solutions in the subcritical case $2p < 2N/(N-2)$ has been the object of intensive research in the last ten years, starting from \cite{ AmbrosettiColorado,LinWeiErratum, LinWei,MaiaMontefuscoPellacci}. By now, a good description of positive solutions and of least energy solutions is available, and we refer for instance to the introductions of the recent papers \cite{Soave, SoaveTavares} for more details. One of the interesting features of these systems  in the fact that they admit solutions with trivial components; for this reason, the systems are sometimes called weakly coupled. It should be noted that both the sublinear/superlinear character of the exponent $p$  \cite{Mandel, OliveiraTavares} or the number of the equations \cite{CorreiaOliveiraTavares} influence the existence results of nontrivial least energy solutions.

As for the critical case $p=N/(N-2)$, its study is recent, starting from the paper by Chen and Zou \cite{ChenZouARMA2012}, for a system with $m=2$ equations and exponent $p=2$. There, it is proved the existence of $0<\beta_1<\beta_2$ such that the system admits a least energy solution (i.e., a solution with minimal energy among all solutions with nontrivial components) for each $\beta\in (-\infty,0)\cup(0,\beta_0)\cup (\beta_1,+\infty)$. In the case $\lambda_1=\lambda_2$, there are no positive solutions for some ranges of $\beta$. The authors treated the higher dimension case $N\geq 5$ in \cite{ChenZou2}, obtaining the existence of least energy solutions for any $\beta\neq 0$. This shows that, in the critical case, the dimension plays a very importante role. Solutions with one sign-changing component are obtained, in the competitive case, in \cite{ChenLinZouCPDE2014}. Observe also that if $N\geq 5$ then $p=N/(N-2)<2$. This, in general, may bring several complications from a technical point of view to deal with \eqref{eq:systemp} in the critical case: for instance,  the term $u_i |u_i|^{p-2}$ is not of class $C^1$. 

As for concentration and blowup results for the Brezis-Nirenberg problem in systems, we are only aware of the paper by Chen and Lin \cite{ChenLin}, where the authors prove for the $m=2$ equation system the analogue of Han's result \cite{HanAIHP} for the single equation: more precisely, they prove that, for $\beta\in (0,\beta_0)\cup (\beta_1,\infty)$ if $N=4$, and $\beta>0$ if $N\geq 5$, the least energy solutions $(u_{\lambda_1},u_{\lambda_2})$ found by Chen and Zou both concentrate and blowup at the same point $x_0$, which is a critical point of the Robin function, as $\lambda_1,\lambda_2\to 0$.

\medbreak

Here we deal with system \eqref{eq:system_2eq}, in the competitive or weakly cooperative cases. Before we state our main results, we need to introduce some notations. Given $\delta>0$ and $\xi\in \Omega$, let
\[
U_{\delta,\xi}(x)=c_4\frac{\delta}{\delta^2+|x-\xi|^2},\qquad x\in \R^4,
\]
with $c_4=2\sqrt{2}$. These functions correspond to all positive solutions of the critical problem in the whole space:
\begin{equation}\label{eq:equation_wholespace}
-\Delta U=U^3 \qquad \text{ in } \R^4.
\end{equation}
Observe that a straightforward computation shows that $U$ is a solution to \eqref{eq:equation_wholespace} if and only if $\mu^{-1/2}U_{\delta,\xi}$ solves 
\[
-\Delta U=\mu U^3 \qquad \text{ in } \R^4.
\]
Observe that $U_{\delta,\xi}\in \mathcal{D}^{1,2}(\R^4)$ (but not in $H^1(\R^4)$), and we define $P$ as the projection map of $\mathcal{D}^{1,2}(\R^4)$ onto $H^1_0(\Omega)$. Throughout this paper we will deal with $PU_{\delta,\xi}$, which is the unique solution of the problem
\begin{equation}\label{eq:projection_of_U}
-\Delta W=-\Delta U_{\delta,\xi}=U_{\delta,\xi}^3\quad  \text{ in } \Omega,\qquad W=0  \text{ on } \partial \Omega.
\end{equation}

We also denote by $G(x,y)$ the Green function of $(-\Delta,H^1_0(\Omega))$, which satisfies, given $x\in \Omega$,
\[
\begin{cases}
-\Delta_y G(x,y)=\delta_x & \text{ for } y\in  \Omega\\
G(x,y)=0 & \text{ for } y\in  \partial \Omega.
\end{cases}
\]
The Green function can be decomposed as
\[
G(x,y)=\frac{\alpha_4}{|x-y|^2}-H(x,y),\qquad \forall (x,y)\in \Omega^2,\ x\neq y,
\]
where $\alpha_4=(2|\omega_3|)^{-1}$ ($\omega_3$ denotes the measure of the unit sphere $S^3\subset \R^4$) and $H$, the regular part of $G$, satisfies
\[
\begin{cases}
-\Delta_y H(x,y)=0 & \text{ for } y\in  \Omega\\
H(x,y)=\alpha_4/|x-y|^2 & \text{ for } y\in  \partial\Omega.
\end{cases}
\]
The Robin function $\tau:\Omega\to \R$ is defined by $\tau(x):=H(x,x)$. It is well know that $\tau$ is a $C^2(\Omega)$ positive function and that $\tau(x)\to +\infty$ as $\dist(x,\partial \Omega)\to 0$, hence $\tau$ has always a (positive) minimum in $\Omega$.

Finally, we will denote the standard inner product and norm in $H^1_0(\Omega)$ by:
\[
\langle u,v \rangle_{H^1_0}=\int_\Omega \nabla u\cdot \nabla v,\qquad \|u\|_{H^1_0}:=\left(\int_\Omega |\nabla u|^2\right)^{1/2},
\]
and the $L^p(\Omega)$ norms by $\|\cdot \|_p$, for $p\geq 1$.

\medbreak

Let $\lambdab:=(\lambda_1,\dots,\lambda_m).$ Our first main result deals with a situation where the Robin function admits $m$ local minimums.

\begin{theorem}\label{thm:main1}
Assume there exist $m$ mutually disjoint open sets $\Lambda_1,\ldots, \Lambda_m$ such that $\inf_{\Lambda_i} \tau<\inf_{\partial \Lambda_i} \tau$.  Then there exist $\bar \beta>0$ and points $\{\xi_1^0,\ldots, \xi_m^0\}$, with $\xi_i^0\in \Lambda_i$ and $\tau(\xi_i^0)=\min_{\Lambda_i} \tau$, such that, for every $-\infty<\beta\leq \bar \beta$ fixed, there exists a family of solutions $(u_{1}^{\lambdab},\ldots, u_m^{\lambdab})$ of \eqref{eq:system_2eq} such that $u_i^{\lambdab}$ concentrates at $\xi_{i}^0$ as $\eps:=\max\{\lambda_1,\ldots,\lambda_m\}\to 0$.

More precisely, there exists    $\delta_{i}^{\lambdab}:=e^{-\frac{d_{i}^{\lambdab}}{\lambda_i}}\to 0$ (with $d_{i}^{\lambdab}\to \frac{c_4}{\omega_3}\left(\int_{\R^4}U_{1,0}^3\right)^2 \tau(\xi_i^0)$) and points $\xi_{i}^{\lambdab}\to \xi_i^0$ such that
\[
u_i^{\lambdab}=\mu_i^{-1/2}PU_{\delta_{i}^{\lambdab},\xi_{i}^{\lambdab}}+\phi_{i}^{\lambdab},\qquad \text{ with }\qquad \|\phi_{i}^{\lambdab}\|_{H^1_0}\to 0 \quad \text{ as } \lambdab\to 0.
\]
\end{theorem}

In the previous result, concentration holds at local minimums. In case of non strict minimizers,  we cannot prescribe the points of concentration.  In the next main result, asking a stronger condition on some critical points of $\tau$, we are able to prescribe the concentration points. We start by recalling the following definition (see \cite{YYLi}, \cite[Definition 2.4]{MussoPistoiaIndiana2002}).

\begin{definition}\label{def:stable}
We say that $x_0\in \Omega$ is a $C^1$--\emph{stable} critical point of $\tau$ if $\nabla \tau(x_0)=0$, and there exists $U$, a neighborhood of $x_0$, such that
\begin{enumerate}
\item[(i)] $\nabla \tau(x)\neq 0$ for every $x\in \partial U$;
\item[(ii)] $\nabla \tau(x)=0$ for some $x\in U$ if, and only if, $\tau(x)=\tau(x_0)$;
\item[(iii)] $\textrm{deg} (\nabla \tau,U,0)\neq 0$, where $deg$ denotes the Brower degree.
\end{enumerate}
\end{definition}

Observe that  scrict local minimums/maximums of $\tau$, or degenerate critical points of $\tau$ are examples of $C^1$--stable critical points. 

\begin{theorem}\label{thm:main2}
Assume that the Robin function $\tau$ admits $m$ stable critical points: $\{\xi_1^0,\ldots, \xi_m^0\}$. Then there exists $\bar \beta>0$ such that for every $-\infty<\beta\leq \bar \beta$ fixed there exists a family of solutions $(u_{1}^{\lambdab},\ldots, u_m^{\lambdab})$ of \eqref{eq:system_2eq} such that $u_i^{\lambdab}$ concentrates at $\xi_{i}^0$ as $\eps:=\max\{\lambda_1,\ldots,\lambda_m\}\to 0$, in the sense of Theorem \ref{thm:main1}.
\end{theorem}

We are also able to prove concentration as both $\beta\to -\infty$ and $\lambda_1,\ldots,\lambda_m\to 0$, for some compatible velocities of these parameters.

\begin{theorem}\label{thm:main3}
Assume that the Robin function $\tau$ admits either $m$ disjoint sets $\Lambda_1,\ldots,\Lambda_m$ with  $\inf_{\Lambda_i} \tau<\inf_{\partial \Lambda_i} \tau$, or that admits $m$ stable critical points. Let $\{ \xi_1^0,\ldots, \xi_m^0\}$ be such that $\xi_i^0\in \Lambda_i$ minimizes $\tau$ in $\Lambda_i$ in the first situation, or that $\xi_i^0$ is a $C^1$--stable critical point of $\tau$ in the second. Then if we take any function $\beta=\beta(\lambdab)<0$ satisfying (for some $\delta>0$)
\[
\beta= \textrm{o}\left(\exp\left(\frac{1}{2\lambda_i}\frac{c_4}{\omega_3}\left(\int_{\R^4}U_{1,0}^3\right)^2 \tau(\xi_i^0)  \right)\right) ,\qquad \forall i=1,\ldots, m,
\] then,  as $\lambdab\to 0$, the conclusions of Theorems \ref{thm:main1} and \ref{thm:main2} hold.
\end{theorem}

\medbreak

Theorems \ref{thm:main1} and \ref{thm:main2} seem to be the first concentration results in the literature to deal with the competitive case. Due to this type of interaction, it is somehow natural to obtain solutions concentrating at different points of the domain. It is interesting to note that, combining these results with the one of Chen and Lin \cite{ChenLin}, we conclude that in the case of systems with two equations, if the Robin function of $\Omega$ admits two minimizers or two $C^1$--stable critical points, then we have the existence of two families of concentrating positive solutions: the least energy ones (considered by Chen and Lin) which concentrate at the same point, and the ones we obtain in this paper, which concentrate at two different points.

Theorem \ref{thm:main3} tells us that the case $\beta\to -\infty$ satisfies the same properties of the case $\beta<0$ fixed, as long as $\beta$ diverges with a ``sufficiently slow velocity'', as $\lambdab\to 0$. Observe that letting $\beta\to -\infty$, generally speaking, induces a segregation phenomen (since the competition increases). Actually, to be more precise, it is known that, for each $\lambdab$ fixed, if a family of solutions $(u_\beta)$ is uniformly bounded in $\beta$ for the $L^\infty$--norm, then $\|u_\beta\|_{C^{0,1}}$ is also bounded uniformly in $\beta<0$ (see \cite[Theorem 1.3]{SoaveZilio}). Thus, up to a subsequence, $u_\beta \to \bar u$ in $C^{0,\alpha}$ for every $\alpha\in (0,1)$, as $\beta\to -\infty$. Moreover, it is known that the possible limiting configurations $\bar u=(\bar u_1,\ldots, \bar u_m)$ satisfy $\bar u_i\cdot \bar u_j\equiv 0$ in $\Omega$ (a phenomenon called segregation), and that
\begin{equation}\label{eq:BN_at_the_sametime}
-\Delta \bar u_i = \mu_i \bar u_i^3 + \lambda_i \bar u_i^3 \qquad \text{ in the open set } \{u_i\neq 0\}
\end{equation}
(check for instance \cite[Theorem 1.4]{NTTV} for the case $N\leq 3$, or \cite[Theorem 1.5]{STTZ} for the general case). The $L^\infty$--bounds are generally fulfilled if the energy levels of the solutions have appropriate variational characterization. This is for instance the case with least energy solutions (check \cite[Lemma 6.1]{ChenZouARMA2012}). Thus, it is expected that, at the limit, one finds $m$ disjoint domains forming a partition of $\Omega$ and, on each of the subdomains, a Brezis-Nirenberg type problem \eqref{eq:BN_at_the_sametime} for the single equation. However, it is (in the critical case) difficult to establish under which conditions does one have $\bar u_i\not \equiv 0$ for all $i$, even for least energy solution (see Theorem 1.4 and Remark 1.3 in \cite{ChenZouARMA2012}). Theorem \ref{thm:main3} seems to be the first result where an asymptotic study is performed when both parameters $\lambda$ and $\beta$ are moving at the same time. In light of what was said in this paragraph, it does not seem obvious if one could have obtained a concentration result by letting first $\beta\to -\infty$, and afterwards $\lambdab\to 0$.

\medbreak

Let us now see some applications of these results to particular domains. 

\begin{example}[Dumbbell--type domains]
Consider any $m$ mutually disjoint connected open sets $\Omega_i$, and let $\Omega_0:=\Omega_1\cup \ldots\cup \Omega_m$. Let $\Omega_\rho$ be a set consisting of $\Omega_1,\ldots, \Omega_m$ connected by $\rho$--thin handles, that is, a connected set $\Omega_\rho=(\cup_{i=1}^m \Omega_i) \cup (\cup_{i=1}^lD_i)$ with $l\geq n-1$ where, for each $i$, $D_i$ is contained in a bounded cylinder (up to a rotation and translation) of the form $\{x=(x',x_n): |x'|\leq \rho, |x_N|\leq {\rm diam}(\Omega_0) \}$. By \cite[Lemma 3.2]{MussoPistoiaIndiana2002}, we know that as $\rho\to 0$, $\tau_{\Omega_\rho}(x)\to \tau_{\Omega_0}(x)$ $C^1$--uniformly in compact sets of $\Omega_0$. Thus, since in each $\Omega_i$ the Robin function admits a minimum, for sufficiently small $\rho$ we have that $\inf_{\Omega_i} \tau_{\Omega_\rho}<\inf_{\partial \Omega_i}\tau_{\Omega_\rho}$. Thus, given a system \eqref{eq:system_2eq} with $m$ equations, the conclusions of Theorems \ref{thm:main1} and \ref{thm:main3} hold on $\Omega_\rho$.
\end{example}

\begin{example}[Domains with holes]
Let us check that, on almost all domains with holes, $\tau$ admits at least two non-degenerate critical points, so that we can apply Theorems \ref{thm:main2}--\ref{thm:main3}, constructing families  of solutions of 2--equation systems which blowup and concentrate at two different points.

 In fact, let $\Omega$ be a bounded domain, and $S\Subset \Omega$ a domain with $\partial S\in C^k$, with $k\geq 4$. Let $\mathcal{C}^k$ be the set of $\theta: \R^N\to \R^N$ of class $C^k$ and such that $\theta=Id$ in a neighborhood of $\partial \Omega$ and $\|\theta\|_{C^k}<\infty$. Then \cite[Theorem 1.1]{MichelettiPistoiaPotential} implies that there exists $\delta$ such that the set $\{\Omega_\theta:=(Id+\theta)(\Omega\setminus S):\ \theta\in \mathcal{C}^k, \ \|\theta\|_{{C}^k}\leq  \delta, \textrm{ and all critical points of the Robin function of } \Omega_\theta \textrm{ are nondegenerate}\}$ is residual in the set $\{\Omega_\theta:=(Id+\theta)(\Omega\setminus S):\ \theta\in \mathcal{C}^k, \ \|\theta\|_{{C}^k}\leq  \delta\}$. Observe that all small perturbations of $\Omega \setminus S$ have Lusternik-Schnirelmann category at least 2, so that the corresponding Robin's function has at least two critical points.
\end{example}

\medbreak

The proof of the main theorems use a classical Lyapunov-Schmidt procedure. The assumption $-\infty<\beta\leq \bar \beta$ for $\bar \beta>0$ small will play a key role in the proof of the reduction method, more precisely to obtain inequalities \eqref{eq:B'_n1}--\eqref{eq:B'_n2} ahead. Moreover, it allows to prove that the solutions we obtain are positive. 

\medbreak

The structure of this paper is as follows. In Section \ref{sec:Preliminaries} we set the framework to perform the finite reduction, which is then the object of Section \ref{sec:Reduction}. Here we will follow the structure of \cite{MussoPistoiaIndiana2002}. With respect to the single equation, we need to take into account the interaction term $\int_\Omega \beta u_i^2 u_j^2$; moreover, since we are dealing with $N=4$, some extra care is needed, since the limiting profiles $U_{\delta,\xi}$ do not belong to $L^2(\R^4)$. Once we have reduced the problem to a finite dimensional one, we prove a $C^1$-energy expansion of the reduced functional in Section \ref{sec:EnergyExpansion}. Observe that we need to deal with the energy functional, and not only with the system, since in Theorem \ref{thm:main1} we deal with minimisers, and not only with critical points. Finally, we conclude the proofs of Theorem \ref{thm:main1}, \ref{thm:main2} and \ref{thm:main3} in Section \ref{sec:Proofs}. We have collected some useful results in Appendix \ref{sec:Appendix}.

\section{The reduction argument: preliminaries}\label{sec:Preliminaries}

In order to find blowup solutions of \eqref{eq:system_2eq}, we will use a Lyapunov-Schmidt reduction method. Since we are interested in positive solutions, we will deal with the alternative system
\begin{equation}\label{eq:system_2eq_2}
-\Delta u_i=\mu_i (u_i^+)^3 + \beta u_i \sum_{j\neq i} u_j^2+\lambda_i u_i,
\end{equation}
We are interested in a situation where $\lambda_i$ is close to zero, and so we consider from now on that $0<\lambda_i<\lambda_1(\Omega)$, where $\lambda_1(\Omega)$ denotes the first eigenvalue of $(-\Delta,H^1_0(\Omega))$. It is easy to check that, for $\beta<0$, nontrivial solutions of  \eqref{eq:system_2eq_2} are positive. This is not the case, in general, for $0<\beta<\bar \beta$, however we will be able to construct solutions of \eqref{eq:system_2eq_2} which are positive, if $\bar \beta>0$ (we refer to the end of the proof of Theorem \ref{thm:main1} for more details).

Solutions of \eqref{eq:system_2eq} correspond to critical points of the functional $E:H^1_0(\Omega,\R^m)\to \R$ defined by
\[
E(u_1,\ldots, u_m)=\sum_{i=1}^m \int_\Omega \left( \frac{1}{2}|\nabla u_i|^2-\frac{\lambda_i}{2}u_i^2-\frac{\mu_i}{4} (u_i^+)^4 \right)-\frac{\beta}{2} \mathop{\sum_{i,j=1}^m}_{i<j}\int_\Omega u_i^2u_j^2.
\]

\medbreak

We take $i^*:L^{4/3}(\Omega)\to H^1_0(\Omega)$, the adjoint operator of the embedding $i:H^1_0(\Omega)\hookrightarrow L^{4}(\Omega)$, which can be characterised by
\[
i^*(u)=v \iff \begin{cases}
-\Delta v=u\\
v\in H^1_0 (\Omega)
\end{cases}
\] 
This adjoint operator is continuous, namely there exists a constant $c>0$ such that
\begin{equation}\label{eq:i*continuous}
\|i^*(u)\|_{H^1_0}\leq c \|u \|_{4/3}\qquad \forall u\in L^{4/3}(\Omega).
\end{equation}

Using the operator $i^*$, we rewrite system \eqref{eq:system_2eq_2} as follows:
\[
u_i= i^*\left(\mu_i f(u_i)+\beta u_i \sum_{j\neq i} u_j^2 + \lambda_i u_i\right),\\
\]
where $f:\R\to \R$ is defined by $f(s)=(s^+)^3$. We will look for solutions of \eqref{eq:system_2eq} of the form
\[
u_i=\mu_i^{-1/2}PU_{\delta_i,\xi_i}+\phi_i,
\]
so that 
\begin{multline}\label{eq:system_i*1}
\mu_i^{-1/2}PU_{\delta_i,\xi_i}+\phi_i = i^*\left[\mu_i f(\mu_i^{-1/2}PU_{\delta_i,\xi_i}+\phi_i)\right.\\
\left.+\beta (\mu_i^{-1/2}PU_{\delta_i,\xi_i}+\phi_i) \sum_{j\neq i}(\mu_j^{-1/2}PU_{\delta_j,\xi_j}+\phi_j)^2 + \lambda_i (\mu_i^{-1/2}PU_{\delta_i,\xi_i}+\phi_i)\right]
\end{multline}
for $i=1,\ldots, m$. Observe that the unknowns in this system are $\delta_i$, $\xi_i$, $i=1,\ldots, m$, and the remainder terms $\phi_i\in H^1_0(\Omega)$.  In order to proceed with the reduction, we need to split the space $H^1_0(\Omega)$ into the sum of two spaces, one of which having finite dimension.

Given $\delta>0$ and $z\in \Omega$, we define
\[
\psi^0_{\delta, z}(x):=\delta \frac{\partial U_{\delta,z}}{\partial \delta}(x)=c_4 \delta\frac{|x-z|^2-\delta^2}{(\delta^2+|x-z|^2)^2}
\]
and, for $j=1,\dots, 4$,
\[
\psi^j_{\delta,z}(x):=\delta \frac{\partial U_{\delta,z}}{\partial z_j}(x)=2c_4\,\delta^2 \frac{x_j-z_j}{(\delta^2+|x-z|^2)^2}.
\]
These functions span the set of solutions of the linearized equation of \eqref{eq:equation_wholespace}, namely
\[
-\Delta \psi=3U^2_{\delta,z}\psi\quad \text{ in } \R^4
\]
(see \cite[Lemma A.1]{BianchiEgnell}), and have bounded $L^4$ norms:
\[
\int_\Omega (\psi^0_{\delta,z})^4=c_4^4 \int_{\frac{\Omega-z}{\delta}}\frac{(|y|^2-1)^4}{(1+|y|^2)^8},\quad \text{ and } \quad \int_\Omega (\psi^j_{\delta,z})^4=(2c_4)^4 \int_{\frac{\Omega-z}{\delta}} \frac{y_j^4}{(1+|y|^2)^8} \ \text{ for } j=1,\ldots, 4.
\]

Define
\[
\Lambda=\left\{ (\deltab,\xib):\ \deltab=(\delta_1,\ldots,\delta_m)\in (\R^+)^m,\ \xib=(\xi_1,\ldots,\xi_m)\in \Omega^m,\ \xi_i\neq \xi_j\ \forall i\neq j  \right\}.
\]
Given $(\deltab,\xib)\in \Lambda$ and $i\in \{1,\ldots, m\}$, we define
\[
K_i=K_{\delta_i,\xi_i} =\spann \left\{P \psi^j_{\delta_i,\xi_i},\  j=0,\ldots, 4 \right\} =P \Ker(-\Delta-3U_{\delta_i,\xi_i}^2) 
\]
and its orthogonal space
\[
K_i^\perp=K_{\delta_i,\xi_i}^\perp=\left\{ \phi\in H^1_0(\Omega):\ \langle \phi, P \psi^j_{\delta i,\xi_i}\rangle_{H^1_0}=0 \ \forall  j=0,\ldots, 4 \right\}.
\]
Moreover, we take
\[
K_{\deltab,\xib}=K_1\times\ldots\times K_m,\qquad K^\perp_{\deltab,\xib}=K_1^\perp\times \ldots\times K_m^\perp.
\]
We will also consider the projection maps 
\[
\Pi_i=\Pi_{\delta_i,\xi_i}: H^1_0(\Omega)\to K_{\delta_i,\xi_i},\qquad \Pi^\perp_i=\Pi^\perp_{\delta_i,\xi_i}:H^1_0(\Omega)\to K^\perp_{\delta_i,\xi_i},
\]
and
\begin{align*}
&\Pi_{\deltab,\xib}=\Pi_1\times\ldots\times \Pi_m:H^1_0(\Omega,\R^m)\to K_{\deltab,\xib},\\
& \Pi_{\deltab,\xib}^\perp=\Pi_1^\perp\times\ldots\times \Pi_m^\perp:H^1_0(\Omega,\R^m)\to K^\perp_{\deltab,\xib}.
\end{align*}

The $m$--equation system \eqref{eq:system_i*1} is equivalent to the system of $2m$ equations

\begin{multline}\label{eq:system_i*1K}
\Pi_i\left\{ \mu_i^{-1/2}PU_{\delta_i,\xi_i}+\phi_i \right\} = \Pi_i  \circ i^*  \left[\mu_i f(\mu_i^{-1/2}PU_{\delta_i,\xi_i}+\phi_i)\right. \\
\left.+\beta (\mu_i^{-1/2}PU_{\delta_i,\xi_i}+\phi_i)\sum_{j\neq i} (\mu_j^{-1/2}PU_{\delta_j,\xi_j}+\phi_j)^2 + \lambda_i (\mu_i^{-1/2}PU_{\delta_i,\xi_i}+\phi_i)\right]
\end{multline}

\begin{multline}\label{eq:system_i*1Kperp}
\Pi^\perp_i\left\{ \mu_i^{-1/2}PU_{\delta_i,\xi_i}+\phi_i \right\} = \Pi^\perp_i \circ i^*\left[\mu_i f(\mu_i^{-1/2}PU_{\delta_i,\xi_i}+\phi_i)\right. \\
\left.+\beta (\mu_i^{-1/2}PU_{\delta_i,\xi_i}+\phi_i)\sum_{j\neq i} (\mu_j^{-1/2}PU_{\delta_j,\xi_j}+\phi_j)^2 + \lambda_i (\mu_i^{-1/2}PU_{\delta_i,\xi_i}+\phi_i)\right] 
\end{multline}
which we will solve for $(\deltab,\xib)\in \Lambda$, and $(\phi_1,\ldots,\phi_m)\in K^\perp_{\deltab,\xib}$.
From system \eqref{eq:system_i*1Kperp}, we will obtain $\phib=(\phi_1,\ldots,\phi_m)$ as a function of $\deltab$ and $\xib$, reducing the problem to a finite-dimensional one. 

\section{Reduction to a finite dimensional problem}\label{sec:Reduction}

We now focus on \eqref{eq:system_i*1Kperp}, which we rewrite as
\[
L_{\deltab,\xib}(\phib)=R_{\deltab,\xib}+N_{\deltab,\xib}(\phib),
\]
where $R_{\deltab,\xib}=(R^1_{\deltab,\xib},\ldots,R^m_{\deltab,\xib})\in K^\perp_{\deltab,\xib}$ is given by ($i\in \{1,\ldots, m\}$) 
\[
R_{\deltab,\xib}^i=\Pi_i^\perp \circ i^*\left\{ \mu_i^{-1/2}((PU_{\delta_i,\xi_i})^3-U_{\delta_i,\xi_i}^3)+\lambda_i \mu_i^{-1/2}PU_{\delta_i,\xi_i} + \beta \mu_i^{-1/2}PU_{\delta_i,\xi_i} \sum_{j\neq i} \mu_j^{-1}(PU_{\delta_j,\xi_j})^2 \right\}, 
\]
where we have used the fact that $PU_{\delta_i,\xi_i}=i^* (U_{\delta_i,\xi_i}^3)$, as well as the positivity of the functions $U_{\delta_i,\xi_i}$ and $PU_{\delta_i,\xi_i}$. Moreover, the operators 
\[
L_{\deltab,\xib}=(L_{\deltab,\xib}^1,\ldots,L_{\deltab,\xib}^m),N_{\deltab,\xib}=(N_{\deltab,\xib}^1,\ldots,N_{\deltab,\xib}^m):K_{\deltab,\xib}^\perp\to  K_{\deltab,\xib}^\perp
\]
are defined by ($i\in \{1,\ldots, m\}$)
\begin{multline}\label{eq:def_of_L}
L_{\deltab,\xib}^i(\phib)=\Pi_{i}^\perp \left\{ \phi_i - i^* \left[ 3(P U_{\delta_i,\xi_i})^2\phi_i + \lambda_i \phi_i  + 2 \beta \mu_i^{-1/2} PU_{\delta_i,\xi_i} \sum_{j\neq i} \mu_j^{-1/2} (PU_{\delta_j,\xi_j})\phi_j  \right.\right.\\
\left.\left. 	    +\beta \phi_i \sum_{j\neq i} \mu_j^{-1}(P U_{\delta_j,\xi_j})^2 \right]  \right\}
\end{multline}
and
\begin{multline}\label{eq:N}
N^i_{\deltab,\xib}(\phib)=\Pi_{i}^\perp \circ i^*\left[ \mu_i\left( f(\mu_i^{-1/2}PU_{\delta_i,\xi_i}+\phi_i)-f(\mu_i^{-1/2}PU_{\delta_i,\xi_i})-f'(\mu_i^{-1/2}PU_{\delta_i,\xi_i})\phi_i\right)  \right.\\
		\left.   + \beta \mu_i^{-1/2} (PU_{\delta_i,\xi_i})\sum_{j\neq i} \phi_j^2 + \beta \phi_i \sum_{j\neq i}\phi_j^2 + 2\beta \phi_i \sum_{j\neq i} \mu_j^{-1/2}(P U_{\delta_j,\xi_j})\phi_j \right]
\end{multline}

Given $\eta$ small, define the set:
\[
X_{\eta}=\left\{ \xib=(\xi_1,\xi_2,\xi_3,\xi_4):\ \dist(\xi_i,\partial \Omega)\geq \eta \ \forall i,\ |\xi_i- \xi_j |\geq \eta \ \forall i\neq j  \right\}.
\]

\begin{lemma}\label{lemma:L_invertible}
There exists $\bar \beta>0$ such that, for every $-\infty<\beta\leq \bar \beta$ and every $\eta>0$ small there exist $\eps_0>0$ and $c>0$ such that, whenever $0<\lambda_i,\delta_i <\eps_0$ for every $i$, and $\xib \in X_\eta$, it holds
\[
\|L_{\deltab,\xib}(\phib)\|_{H^1_0} \geq c\| \phib\|_{H^1_0} \qquad \forall \phib \in K^\perp_{\deltab,\xib}.
\]
Furthermore, $L_{\deltab,\xib}$ is an invertible operator with continuous inverse.
\end{lemma}
\begin{proof}
We will follow the structure of \cite[Lemma 1.7]{MussoPistoiaIndiana2002}. The main differences here are the fact that we work in dimension 4 (while in \cite{MussoPistoiaIndiana2002} the dimension is 5 or higher, which simplifies the computations and estimates of some integral terms), and the presence of the competition terms between components. Fix $\beta<0$,  Assume, in view of a contradiction, that there exists:
\begin{itemize}
\item $\eta>0$ and sequences $\delta_{in},\lambda_{in}\to 0$, $\xi_{in}\to \xi_i\in X_\eta$ such that $\xib_n=(\xi_{1n},\xi_{2n},\xi_{3n},\xi_{4n})\in X_\eta$, 
\item sequences $\phi_{in}\in K_i^\perp$ such that $\|\phib_n \|_{H^1_0}=1$ and $\|L_{\deltab_n,\xib_n}(\phib_n)\|_{H^1_0}=:\|\mathbf{h}_n\|_{H^1_0}\to 0$.
\end{itemize}

Observe that we have
\begin{multline}\label{eq:equation_with_h_and_w}
\phi_{in}=i^* \left[ 3(PU_{in})^2\phi_{in}+\lambda_{in}\phi_{in}+2\beta \mu_i^{-1/2}(PU_{in}) \sum_{j\neq i} \mu_{j}^{-1/2}(PU_{jn})\phi_{jn}+\beta \phi_{in} \sum_{j\neq i} \mu_j^{-1}(PU_{jn})^2 \right]\\
								+h_{in}+w_{in}
\end{multline}
with $w_{in}\in K_i$. Here we have denoted $U_{in}:=U_{\delta_{in},\xi_{in}}$. We will also make the identification $\psi^k_{in}:=\psi^k_{\delta_{in},\xi_{in}}$, $k=0,\ldots, 4$.

\medbreak

\noindent \emph{Step 1.} Let us check that $w_{in}\to 0$ in $H^1_0$, as $n\to +\infty$.

We have
\[
w_{in}=\sum_{k=0}^4 c_{in}^k (P\psi_{in}^k),
\]
for some coefficients $c_{in}^k$. Thus, taking in consideration Lemma \ref{lemma:expansion_of_PU} (identities \eqref{eq:expansion2} and \eqref{eq:expansion3}), and the definition of projection, we have for every $k,j\in \{0,1,\ldots, 4\}$:
\[
\inner{P\psi^j_{in}}{P\psi_{in}^k}=\int_\Omega \nabla P\psi^j_{in}\cdot \nabla P\psi^k_{in}=\int_\Omega 3(U_{in})^2 \psi_{in}^j P\psi_{in}^k=\int_\Omega 3(U_{in})^2 \psi^j_{in}\psi^k_{in}+\ob_n(1).
\]
After a change of variables $y=\xi_{in}+\delta_{in}x$, one can easily check that the integral $\int_\Omega (U_{in})^2 \psi^j_{in}\psi^k_{in}\to \sigma_{jk}$  as $n\to \infty$, where: 
\[
\sigma_{00}:=\int_{\R^4} \frac{c_4^4(|y|^2-1)^2}{(1+|y|^2)^6}\, dy>0,\quad \sigma_{jj}:=\int_{\R^4} \frac{4c_4^4 y_j^2}{(1+|y|^2)^6}\, dy>0\quad  (j=1,\ldots, 4),
\]
and
\[
\begin{split}
&\sigma_{0j}=\sigma_{j0}:=\int_{\R^4} \frac{2c_4^4 (|y|^2-1)y_j}{(1+|y|^2)^6}\, dy=0 \quad (j\geq 1),\\
&\sigma_{jk}=\sigma_{kj}:=\int_{\R^4} \frac{4c_4^4 y_jy_k}{(1+|y|^2)^6}\, dy=0 \quad (j,k\geq 1,\ j\neq k).
\end{split}
\]
In conclusion, we have that
\begin{equation}\label{eq:w_{in}^2=}
\|w_{in}\|_{H^1_0}^2=\sum_{j,k=0}^4 c_{in}^j c_{in}^k\inner{P\psi_{in}^j}{P\psi_{in}^k}=\sum_{k=0}^4 (c_{in}^k)^2\sigma_{kk}+\mathop{\sum_{j,k=0}^4}_{j\neq k} c_{in}^j c_{i n}^k \ob_n(1).
\end{equation}

On the other hand, using equation \eqref{eq:equation_with_h_and_w},
\begin{multline*}
\|w_{in}\|^2_{H^1_0}=\underbrace{-\int_\Omega 3 (PU_{in})^2\phi_{in} w_{in}}_{\textrm{I}_n} -\lambda_{in} \int_\Omega \phi_{in} w_{in}\underbrace{-\beta\int_\Omega \phi_{in} w_{in}\sum_{j\neq i} \mu_j^{-1}(PU_{jn})^2}_{\textrm{II}_n}\\
			\underbrace{-2\beta\int_\Omega \mu_i^{-1/2}(PU_{in})w_{in} \sum_{j\neq i} \mu_{j}^{-1/2}(PU_{jn})\phi_{jn}}_{\textrm{III}_n}+\inner{\phi_{in}}{w_{in}}-\inner{h_{in}}{w_{in}}.
\end{multline*}
Observe that $\inner{h_{in}}{w_{in}}=0$ and that
\[
0=\inner{\phi_{in}}{w_{in}}=\sum_{k=0}^4 c_{in}^k \inner{\phi_{in}}{P\psi^k_{in}}=\sum_{k=0}^4c_{in}^k \int_\Omega (U_{in})^2 \psi_{in}^k \phi_{in}.
\]
By using this last inequality, we see that
\begin{equation}
\begin{split}\label{eq:I_n}
\textrm{I}_n&=3\int_\Omega \left((U_{in})^2-(PU_{in})^2\right)\phi_{in}w_{in}-3\int_\Omega (U_{in})^2 \phi_{in} w_{in}\\ 
	&=3\int_\Omega \left((U_{in})^2-(PU_{in})^2\right)\phi_{in}w_{in}+\sum_{k=0}^4 3c_{in}^k \int_\Omega (U_{in})^2\phi_{in} \left(\psi_{in}^k-P\psi_{in}^k\right).
\end{split}
\end{equation}

Moreover,
\[
\textrm{II}_n=-\beta \sum_{k=0}^4 c_{in}^k  \int_\Omega  (P\psi^k_{in})\phi_{in}  \sum_{j\neq i} \mu_j^{-1}(PU_{jn})^2
\]
and
\[
\textrm{III}_n=-2\beta \mu_i^{-1/2} \sum_{k=0}^4 c_{in}^k \int_\Omega   (P\psi^k_{in})(PU_{in})\sum_{j\neq i}\mu_j^{-1/2}(PU_{jn})\phi_{jn}.
\]
Thus we have
\[
\begin{split}
\|w_{in}\|_{H^1_0}^2\leq & 3  \left\| (U_{in})^2-(PU_{in})^2\right\|_2\|\phi_{in}\|_4\|w_{in}\|_4 +3\sum_{k=0}^4 |c_{in}^k| \|U_{in}\|_4^2 \|\phi_{in}\|_4^2 \left\|\psi_{in}^k-P\psi_{in}^k \right\|_4\\
				& +\lambda_{in} \|\phi_{in}\|_2\|w_{in}\|_2+|\beta| \sum_{k=0}^4 |c_{in}^k| \|\phi_{in}\|_4\sum_{j\neq i} \mu_j^{-1} \left\| (PU_{jn})^2P\psi_{in}^k \right\|_{4/3}\\
				&+2|\beta| \mu_{i}^{-1/2} \sum_{k=0}^4  |c_{in}^k|  \sum_{j\neq i} \mu_j^{-1/2} \|\phi_{jn}\|_4 \left\|(PU_{in})(PU_{jn})P\psi^k_{in} \right\|_{4/3}.
\end{split}
\]
Lemma \ref{lemma:expansion_of_PU} yields that $\left\| (U_{in})^2-(PU_{in})^2\right\|_2\to 0$ and $\left\|\psi_{in}^k-P\psi_{in}^k \right\|_4\to 0$ as $n\to \infty$. Moreover, from Lemma \ref{lemma:auxiliary_lemmas_appendix} we have that  $\left\| (PU_{jn})^2P\psi_{in}^k \right\|_{4/3}=\textrm{O}({\delta_{in}\delta_{jn}})$ and $\left\|(PU_{in})(PU_{jn})P\psi^k_{in} \right\|_{4/3}=\textrm{O}(\delta_{in}\delta_{jn})$. In conclusion, we have
\begin{equation}\label{w_{in}=2nd}
\|w_{in}\|_{H^1_0}^2\leq \ob_n(1) \|w_{in}\|_{H^1_0} +\ob_n(1)\sum_{k=0}^4 |c_{in}^k|.
\end{equation}
Combining this with \eqref{eq:w_{in}^2=}, 
\[
\sum_{k=0}^4 (c_{in}^k)^2 \sigma_{kk}\leq \ob_n(1) \sum_{k=0}^4 |c_{in}^k|+\ob_n(1)\sum_{k=0}^4 (c_{in}^k)^2.
\]
This implies that the sequences $\{c_{in}^k\}$ are bounded in $n$, and consequently, going back to \eqref{w_{in}=2nd}, $w_{in}\to 0$ as $n\to \infty$. This ends the first step of the proof.

\medbreak

\noindent \emph{Step 2.} Define 
\[
\tilde \phi_{in}(y)=\delta_{in} \phi_{in}(\xi_{in}+\delta_{in}y), 
\]
which satisfies
\[
\| (\tilde \phi_{1n},\ldots, \tilde \phi_{mn})\|_{H^1_0\left(\frac{\Omega-\xi_{in}}{\delta_{in}}\right)}=\| (\phi_{1n},\ldots, \phi_{mn})\|_{H^1_0(\Omega)}=1,
\]
so that there exists $\tilde \phi_i$ such that $\tilde \phi_{in}\rightharpoonup \tilde \phi_i$ in $\mathcal{D}^{1,2}(\R^4)$.
The aim of this step is to prove that $\tilde \phi_i\equiv 0$, that is,
\begin{equation}\label{eq:aim_of_step2}
\tilde \phi_{in}\rightharpoonup 0 \qquad \text{ in } \mathcal{D}^{1,2}(\R^4).
\end{equation}

We rewrite \eqref{eq:equation_with_h_and_w}, obtaining 
\begin{multline*}
\tilde \phi_{in}=\delta_{in}^2\, i^*\left[3(PU_{in})^2(\xi_{in}+\delta_{in}\cdot)\tilde \phi_{in} + \lambda_{in} \tilde \phi_{in}  +\beta \tilde \phi_{in} \sum_{j\neq i} \mu_j^{-1} (PU_{jn})^2(\xi_{in}+\delta_{in}\cdot)    + \right.\\
		\left. +2\delta_{in}\beta \mu_i^{-1/2}(PU_{in})(\xi_{in}+\delta_{in}\cdot) \sum_{j\neq i} \mu_j^{-1/2}(PU_{jn})(\xi_{in}+\delta_{in}\cdot)  \phi_{jn}(\xi_{in}+\delta_{in}\cdot)\right] + \tilde h_{in} +\tilde w_{in},
\end{multline*}
with 
\[
\tilde w_{in}(y):=\delta_{in} w_{in}(\xi_{in}+\delta_{in} y),\qquad \tilde h_{in}(y):=\delta_{in} h_{in}(\xi_{in}+\delta_{in} y).
\]
Observe that $\tilde h_{in}\to 0$ and (by Step 1) $\tilde w_{in}\to 0$ strongly in $\mathcal{D}^{1,2}(\R^4)$.

Take $\psi\in \mathcal{C}^\infty_{\rm c}(\R^4)$ and $n$ large so that
\[
K_\psi:=\textrm{supp} \psi \subset \Omega_{in}:= \frac{\Omega-\xi_{in}}{\delta_{in}}.
\]
We have
\[
\begin{split}
\int_{\Omega_{in}}\nabla \tilde \phi_{in}\cdot \nabla \psi= &\underbrace{\delta_{in}^2 \int_{\Omega_{in}} 3 (PU_{in})^2(\xi_{in}+\delta_{in}y) \tilde \phi_{in}(y) \psi(y)}_{A_n} +\delta_{in}^2 \lambda_{in} \int_{\Omega_{in}}\tilde \phi_{in} \psi\\
					 &+\int_{\Omega_{in}}\nabla (\tilde h_{in}+\tilde w_{in})\cdot \nabla \psi  +\underbrace{\delta_{in}^2 \beta \int_{\Omega_{in}} \tilde \phi_{in}(y) \psi(y) \sum_{j\neq i} \mu_j^{-1} (PU_{jn})^2(\xi_{in}+\delta_{in}y)}_{B_n}\\
					 &+\underbrace{2\beta \delta_{in}^3 \mu_i^{-1/2}\int_{\Omega_{in}}(PU_{in})(\xi_{in}+\delta_{in}y)\psi(y) \sum_{j\neq i} \mu_j^{-1/2}(PU_{jn})(\xi_{in}+\delta_{in}y)\phi_{jn}(\xi_{in}+\delta_{in}y)}_{C_n}.
\end{split}
\]
We have $\int_{\Omega_{in}}\nabla(\tilde h_n+\tilde w_{n})\cdot \nabla \psi\to 0$,
\begin{equation}\label{eq_An}
\begin{split}
A_n &=\delta_{in}^2 \int_{\Omega_{in}} 3(U_{in})^2(\xi_{in}+\delta_{in}y) \tilde \phi_{in}(y) \psi(y)+\ob_n(1)=\int_{\Omega_{in}\cap K_\psi} \frac{3c_4^2}{(1+|y|^2)^2}\tilde \phi_{in}(y)\psi(y)+\ob_n(1)\\
	&\to \int_{\R^4} \frac{3c_4^2 }{(1+|y|^2)^2}\tilde \phi_i(y) \psi(y) =\int_{\R^4} 3U_{1,0}^2 \tilde \phi_{i} \psi.
\end{split}
\end{equation}
Moreover,
\begin{equation}\label{eq:B_n}
\begin{split}
B_n&=\delta_{in}^2 \beta \int_{\Omega_{in}} \tilde \phi_{in}(y) \psi(y) \sum_{j\neq i} \mu_j^{-1} (U_{jn})^2(\xi_{in}+\delta_{in} y)+ \beta \ob_n(1)\\
	&=\delta_{in}^2 \delta_{jn}^2\beta \int_{\Omega_{in}\cap K_\psi} \tilde \phi_{in} \psi \sum_{j\neq i} \mu_j^{-1}\frac{c_4^2}{(\delta_{jn}^2+|\delta_{in} y + \xi_{in}-\xi_{jn}|^2)^2}+\beta \ob_n(1)\to 0
\end{split}
\end{equation}
(in fact, $y\in K_\psi$ and $|\xi_{in}-\xi_{jn}|\geq \eta$, so that $(\delta_{jn}^2+|\delta_{in} y + \xi_{jn}-\xi_{in})^2\geq \eta/2$ for large $n$). Analogously,
\begin{equation}\label{eq:C_n}
\begin{split}
C_n&=2\beta  \mu_i^{-1/2}  \delta_{in} c_4  \int_{\Omega_{in}}\frac{ \psi(y)}{1+|y|^2} \sum_{j\neq i} \mu_j^{-1/2} \frac{\delta_{jn} c_4}{\delta_{jn}^2+|\delta_{in}y+\xi_{in}-\xi_{jn}|^2}\delta_{in}\phi_{jn}(\delta_{in}y+\xi_{in}) +\beta \ob_n(1)\\
	&\leq C \beta \delta_{in} \delta_{jn} \|\phi_{jn}\|_{H^1_0(\Omega)} \|U_{1,0}\|_{8/3}\|\psi\|_{8/3}+\beta \ob_n(1)\to 0.
\end{split}
\end{equation}
Thus, at the limit, we find that
\[
-\Delta \tilde \phi_i= 3U_{1,0}^2 \tilde \phi_i \quad \text{ in } \R^4,\qquad \tilde \phi_i \in \mathcal{D}^{1,2}(\R^4).
\]
Let us check that $\tilde \phi_i=0$. For that, it is enough to show that $\tilde \phi_i \in (\Ker(-\Delta-3U_{1,0}^2))^\perp$, that is,
\begin{equation}\label{eq:tilde_phi_orthogonal}
\int_{\R^4} \nabla \tilde \phi_i \cdot \nabla \psi^j_{1,0}=0,\qquad \forall j=0,\ldots, 4.
\end{equation}
Let us check this for $j=0$: we have
\[
0=\inner{\phi_{in}}{P\psi^0_{in}}=\int_\Omega 3 (U_{in})^2 \psi^0_{in} \phi_{in} =\int_{\Omega_{in}}3c_4^3\frac{|y|^2-1}{(1+|y|^2)^4}\tilde \phi_{in}=\int_{\Omega_{in}}3U_{1,0}^2 \psi^0_{1,0} \tilde \phi_{in}.
\]
As $\tilde \phi_{in}\rightharpoonup \tilde \phi_i$ in $L^4(\R^4)$, and $U_{1,0}^2 \psi^0_{1,0}\in L^{4/3}(\R^4)$, then passing to the limit we obtain
\[
\int_{\R^4}3U_{1,0}^2 \psi_{1,0}^0 \tilde \phi_i=0.
\]
This, combined with the equation for $\tilde \phi_i$, implies \eqref{eq:tilde_phi_orthogonal} for $j=0$. The proof for $j=1,\ldots, 4$ is analogous. With this, we have concluded that $\tilde \phi_{in}\rightharpoonup 0$ weakly in $\mathcal{D}^{1,2}(\R^4)$, which was the goal of this step.

\medbreak

\noindent \emph{Step 3.} Let us check that actually $ \phi_{in}\to 0$ strongly in $H^1_0(\Omega)$, for every $i=1,\ldots, m$. This contradicts $\| \phib_n\|_{H^1_0(\Omega)}=1$, and concludes the proof of the first part of the lemma.

Going back to the equation for $\phi_{in}$ (see \eqref{eq:equation_with_h_and_w}), we get
\begin{equation}\label{eq:phi_in}
\begin{split}
\|\phi_{in}\|^2_{H^1_0} =&\underbrace{\int_{\Omega} 3(PU_{in})^2 \phi_{in}^2}_{A'_n}+ \lambda_{in} \int_{\Omega} \phi_{in}^2+\inner{h_{in}+w_{in}}{\phi_{in}}+ \underbrace{\beta \int_{\Omega}  \phi_{in}^2 \sum_{j\neq i} \mu_j^{-1} (PU_{jn})^2}_{B'_n}\\
				&+\underbrace{2\beta \mu_{i}^{-1/2} \int_{\Omega}  (PU_{in}) \phi_{in} \sum_{j\neq i} \mu_j^{-1/2}(PU_{jn})\phi_{jn}}_{C'_n}.
\end{split}
\end{equation}
We have, since $\phi_{in}$ is bounded in $H^1_0(\Omega)$, that $\inner{ h_{in}+ w_{in}}{ \phi_{in}}\to 0$. Moreover,
\[
\lambda_{in}\int_\Omega \phi_{in}^2\to 0.
\]
Since $\tilde \phi^2_{in}\rightharpoonup 0$ in $L^2(\R^4)$ and $(1+|y|^2)^{-2}\in L^2(\R^4)$, then
\[
A_n'=\int_\Omega 3(U_{in})^2\phi_{in}^2+\ob_n(1) =3c_4^2 \int_{\Omega_{in}} \frac{(\tilde \phi_{in})^2(y)}{(1+|y|^2)^2} + \ob_n(1)\to 0.
\]
The term $B'_n$ does not converge to 0, but we can rule it out. In fact, if $\beta<0$, then
\begin{equation}\label{eq:B'_n1}
B'_n= \beta \int_{\Omega}  \phi_{in}^2 \sum_{j\neq i} \mu_j^{-1} (PU_{jn})^2\leq 0.
\end{equation}
If on the other hand $0<\beta\leq \bar \beta$, we have
\begin{equation}\label{eq:B'_n2}
B'_n= \beta \int_{\Omega}  \phi_{in}^2 \sum_{j\neq i} \mu_j^{-1} (PU_{jn})^2 \leq \bar \beta \|\phi_{in}\|_4^2 \sum_{j\neq i} \mu_j^{-1} \|PU_{jn}\|^2_4\leq C \bar \beta \|\phi_{in}\|^2_{H^1_0} \leq \frac{1}{2}\|\phi_{in}\|_{H^1_0}^2,
\end{equation}
for $\bar \beta>0$ sufficiently small.

Finally, the term $C'_n$ can be estimated from above in the following way:
\[
C'_n=2\beta \mu_{i}^{-1/2}(C_n^1+C_n^2+\sum_{j\neq i}\Ob(\delta_{in}\delta_{jn})),
\]
with
\[
\begin{split}
|C_n^1| &= \left| \int_{B_{\frac{\eta}{2}}(\xi_{in})}  (PU_{in}) \phi_{in} \sum_{j\neq i} \mu_j^{-1/2} (PU_{jn}) \phi_{jn} \right|\\
		&\leq \sum_{j\neq i} \mu_j^{-1/2} \left( \int_{B_{\frac{\eta}{2}}(\xi_{in})} (PU_{in})^2(PU_{jn})^2\right)^{1/2}\left(\int_\Omega \phi_{in}^2\phi_{jn}^2\right)^{1/2}\\
		&\leq \sum_{j\neq i} C \delta_{in}\delta_{jn} \left( \int_{B_{\frac{\eta}{2\delta_{in}}}(0)} \frac{1}{(1+|y|^2)^2}\right)^{1/2}=\sum_{j\neq i}\Ob(\delta_{in}\delta_{jn} \sqrt{|\ln \delta_{in}|}).
\end{split}
\]
Analogously,
\[
|C_n^2| =\left| \int_{B_{\frac{\eta}{2}}(\xi_{in})}  PU_{in}(x) \phi_{in} \sum_{j\neq i} \mu_j^{-1/2}  PU_{jn}(x) \phi_{jn} \right| =\sum_{j\neq i}\Ob(\delta_{jn}\delta_{in} \sqrt{|\ln \delta_{jn}|}).
\]
So 
\begin{equation}\label{eq:C'_n}
C'_n=\beta \sum_{j\neq i} \Ob(\delta_{in}\delta_{jn} \sqrt{|\ln \delta_{in}|})\to 0.
\end{equation}
By combining all this with \eqref{eq:phi_in}, we deduce that $\phi_{in}\to 0$ strongly in $H^1_0(\R^4)$, as wanted.

\medbreak

\noindent \emph{Step 4.} (invertibility) The only thing left to prove, at this point, is that $L_{\deltab,\xib}$ is an invertible operator. Since  $i^*$ is a compact operator from  $L^{4/3}(\Omega)$ to $H^1_0(\Omega)$, expression \eqref{eq:def_of_L} readily implies that $L_{\delta_i,\xi_i}=Id-K$, where $K$ is compact. We have already shown that
\begin{equation}\label{eq:inequality_L_invertible_aux}
\|L_{\deltab,\xib}(\phib)\|_{H^1_0}\geq c\| \phib\|_{H^1_0}\qquad \text{ for every } \phib\in K_{\deltab,\xib}^\perp,
\end{equation}
so $L_{\deltab,\xib}$ is injective. Thus, by Fredholm's alternative theorem, it is also surjective. Thus $L_{\deltab,\xib}$ is invertible, and by \eqref{eq:inequality_L_invertible_aux} it follows that $L^{-1}_{\deltab,\xib}$ is continuous. \qedhere

\end{proof}

\begin{proposition}\label{prop:estimate_error}
Let $-\infty<\beta\leq \bar \beta$. Then for every $\eta>0$ small there exist $\eps_0>0$ and $c>0$ such that, whenever $0<\lambda_i,\delta_i <\eps_0$ for all $i$, and $\xib \in X_\eta$, there exists a unique $\phib=\phib^{\deltab,\xib}\in K^\perp_{\deltab,\xib}$ solving
\begin{equation}\label{eq:equation_fixedpoint}
L_{\deltab,\xib}(\phib)=R_{\deltab,\xib}+N_{\deltab,\xib}(\phib),
\end{equation}
and satisfying
\[
\|\phib\|_{H^1_0} \leq c  \sum_{i=1}^m (\lambda_i \delta_i +\delta_i^2)
\]
(for some $c>0$ depending on $\beta$, but independent from $\deltab$ and $\lambdab$).
\end{proposition}

\begin{proof}
From Lemma \ref{lemma:L_invertible}, we have that $L_{\deltab,\xib}$ is invertible, and thus \eqref{eq:equation_fixedpoint} is equivalent to
\[
\phib=L_{\deltab,\xib}^{-1}(R_{\deltab,\xib}+N_{\deltab,\xib}(\phib))=:T_{\deltab,\xib}(\phib).
\]
By combining the continuity of $L^{-1}_{\deltab,\xib}$, the expressions for $R_{\deltab,\xib}$ and $N_{\deltab,\xib}$, and \eqref{eq:i*continuous}, we deduce that
\begin{align*}
\|T_{\deltab,\xib}(\phib)\|_{H^1_0} &\leq c_1\|R_{\deltab,\xib} + N_{\deltab,\xib}(\phib)\|_{H^1_0}  \leq c_2\left(\|\widetilde R_{\deltab,\xib}\|_{4/3}+\|\widetilde N_{\deltab,\xib}(\phib)\|_{4/3}\right)
\end{align*}
with $\widetilde R_{\deltab,\xib}=(\widetilde{R}^1_{\deltab,\xib},\ldots, \widetilde{R}_{\deltab,\xib}^m)\in L^{4/3}(\Omega,\R^m)$ and $\widetilde N_{\deltab,\xib}=(\widetilde{N}^1_{\deltab,\xib},\ldots, \widetilde{N}_{\deltab,\xib}^m):H^1_0(\Omega,\R^m)\to L^{4/3}(\Omega,\R^m)$, given by
\[
\widetilde{R}_{\deltab,\xib}^i= \mu_i^{-1/2}((PU_i)^3-U_i^3)+\lambda_i \mu_i^{-1/2}PU_i +\beta \mu_i^{-1/2}PU_i \sum_{j\neq i} \mu_j^{-1}(PU_j)^2
\]
and
\begin{multline*}
\widetilde{N}^i_{\deltab,\xib}(\phib)=\mu_i\left( f(\mu_i^{-1/2}PU_i+\phi_i)-f(\mu_i^{-1/2}PU_i)-f'(\mu_i^{-1/2}PU_i)\phi_i\right) \\
		  + \beta \mu_i^{-1/2} (PU_i)\sum_{j\neq i} \phi_j^2 + \beta \phi_i \sum_{j\neq i}\phi_j^2 + 2\beta \phi_i \sum_{j\neq i} \mu_j^{-1/2}(P U_j)\phi_j,
\end{multline*}
$i=1,\ldots, m$. Here, we are using the notation $U_i:=U_{\delta_i,\xi_i}$.

\medbreak

\noindent \emph{- Estimate for $\widetilde{N}_{\deltab,\xib}(\phib)$:} First of all observe that there exists $C>0$ such that
\begin{multline*}
\left \|\mu_i\left( f(\mu_i^{-1/2}PU_i+\phi_i)-f(\mu_i^{-1/2}PU_i)-f'(\mu_i^{-1/2}PU_i)\phi_i\right)\right\|_{4/3}\\
\leq C (\|(PU_i) \phi_i^2\|_{4/3} + \|\phi_i^3\|_{4/3}) \leq C \left( \|\phib\|_{H^1_0}^2+\|\phib\|_{H^1_0}^3 \right),
\end{multline*}
where we have used Lemma \ref{lemma:auxiliaryestimate}. Moreover,
\[
\left \|\beta \phi_i \sum_{j\neq i}\phi_j^2\right \|_{4/3}\leq C|\beta| \|\phib\|_{H^1_0}^3
\]
and
\[
\left\| \beta \mu_i^{-1/2} (PU_i)\sum_{j\neq i} \phi_j^2+2\beta \phi_i \sum_{j\neq i} \mu_j^{-1/2}(P U_j)\phi_j \right \|_{4/3} \leq C |\beta| \|\phib\|_{H^1_0}^2,
\]
since $\|PU_i\|_{4}\leq C$. Therefore, there exists $C>0$ such that
\begin{equation}\label{eq:N_deltaxi}
\left\| \widetilde N_{\deltab,\xib}(\phib)\right\|_{4/3}\leq C(1+|\beta|) \left(\|\phib\|_{H^1_0}^2+\|\phib\|_{H^1_0}^3\right).
\end{equation}

\medbreak

\noindent \emph{- Estimate for $\widetilde R_{\deltab,\xib}$:}
Now let us estimate the velocity of $\widetilde{R}_{\deltab,\xib}$. 
\medbreak

Using Lemmas \ref{lemma:expansion_of_PU} and \ref{lemma:auxiliaryestimate}, we have
\begin{align*}
\left \| (PU_i)^3-U_i^3 \right\|^{4/3}_{4/3}&\leq c_1 \int_\Omega \left| U_i^2\Ob(\delta_i)+\Ob(\delta_i^3)\right|^{4/3}\leq \Ob(\delta_i^{4/3}) \int_\Omega U_i^{8/3} +\Ob(\delta_i^4)=\Ob(\delta_i^{8/3}),
\end{align*}
and thus
\[
\left\| (PU_i)^3-U_i^3 \right\|_{4/3}=\textrm{O}(\delta_i^2), \qquad \text{ as } \delta_i\to 0.
\]
Moreover,
\[
\|\lambda_i PU_i\|_{4/3}^{4/3} =\lambda_i^{4/3} \int_\Omega |PU_i|^{4/3}=\lambda_i^{4/3} \int_\Omega U_i^{4/3}+\lambda_i^{4/3} \Ob(\delta_i^{4/3})=\lambda_i^{4/3} \Ob(\delta_i^{4/3}),
\]
so that 
\[
\|\lambda_i PU_i\|_{4/3}=\lambda_i \Ob(\delta_i).
\]
Finally, for $j\neq i$,
\[
\begin{split}
\|(PU_i) (PU_j)^2\|_{4/3}^{4/3} =& \int_\Omega (PU_i)^{4/3}(PU_j)^{8/3}\\
							=&\int_\Omega U_i^{4/3} U_j^{8/3}+   \int_\Omega \left[ (PU_i)^{4/3}-U_i^{4/3}\right] \left[ (PU_j)^{8/3}-U_j^{8/3}\right]  \\
							&+\int_\Omega U_i^{4/3} \left[ (PU_j)^{8/3}-(U_j)^{8/3}\right] +\int_\Omega \left[ (PU_i)^{4/3}-U_i^{4/3}\right]U_j^{8/3}\\
						\leq & \int_\Omega U_i^{4/3} U_j^{8/3} +\int_\Omega  \left| (U_i)^{1/3}\Ob(\delta_i)+\Ob(\delta_i^{4/3})\right| \left| U_j^{5/3} \Ob(\delta_j) +\Ob(\delta_j^{8/3}) \right|\\
							&+\int_\Omega U_i^{4/3} \left| U_j^{5/3}\Ob(\delta_j)+\Ob(\delta_j^{8/3}) \right| +\int_\Omega \left| U_i^{1/3}\Ob(\delta_i)+\Ob(\delta_i^{4/3})\right| U_j^{8/3}\\
							=& \Ob(\delta_i^{4/3}\delta_j^{4/3}),
\end{split}
\]
so that $\|\beta (PU_i) (PU_j)^2\|_{4/3}=|\beta| \Ob(\delta_i\delta_j)$, and
\begin{equation}\label{eq:R_deltaxi}
\left \|\widetilde R_{\deltab,\xib} \right\|_{4/3} =\sum_{i=1}^m (\Ob(\delta_i^2)+\lambda_i \Ob(\delta_i))+\mathop{\sum_{i,j=1}^{m}}_{i\neq j} |\beta| \Ob(\delta_i \delta_j).
\end{equation}

\medbreak

So, combining \eqref{eq:N_deltaxi} and \eqref{eq:R_deltaxi}, we have
\[
\left\| T_{\deltab,\xib}(\phib) \right\|_{H^1_0}\leq C\left( (1+|\beta|)(\|\phib\|_{H^1_0}^2 +\|\phib\|_{H^1_0}^3)+\sum_{i}(\delta_i^2+\lambda_i \delta_i)\right)
\]
and hence, for sufficiently small $\deltab>0$, there exists $c>0$ (depending only on $\beta$) such that, for $0<\lambda_i\leq 1$,
\[
T_{\deltab,\xib}\left(\left\{ \phib\in K_{\deltab,\xib}^\perp:\ \|\phib\|_{H^1_0}\leq c \left(\delta_i^2+\lambda_i\delta_i\right)  \right\}\right)\subseteq \left\{ \phib\in K_{\deltab,\xib}^\perp:\ \|\phib\|_{H^1_0}\leq c \left(\delta_i^2+\lambda_i\delta_i\right)  \right\}.
\]
Using the same arguments as before, we can also check easily that $T_{\deltab,\xib}$ is a contraction, i.e., there exists $L\in (0,1)$ independent of $\deltab$ such that
\[
\|T_{\deltab,\xib}(\psib_1)-T_{\deltab,\xib}(\psib_2)\|_{H^1_0}\leq L \| \psib_1-\psib_2\|_{H^1_0},\quad \forall \psib_1,\psib_2\in \left\{ \phib\in K_{\deltab,\xib}^\perp:\ \|\phib\|_{H^1_0}\leq c \left(\delta_i^2+\lambda_i\delta_i\right)  \right\}.
\]
In conclusion, the proposition follows from the Banach Fixed Point Theorem.
\end{proof}


\begin{lemma}\label{lemma:differentiability}
Let $-\infty<\beta\leq \bar \beta$. For every $\eta>0$ there exist $\eps_0>0$ such that, for $0<\lambda_i<\eps_0$, the map
\[
(0,\eps_0)^m \times X_\eta \to K_{\deltab,\xib}^\perp;\qquad    (\deltab,\xib)\mapsto \phib^{\deltab,\xib}
\]
is continuously differentiable.
\end{lemma}
\begin{proof} Fix $\beta<0$ and $\eta$ small, and take $\eps_0$ as in Proposition \ref{prop:estimate_error}. We apply the implicit function theorem to the map
\[
T: (0,\eps_0)^m \times X_\eta \times K_{\deltab,\xib}^\perp \to K_{\deltab,\xib}^\perp,\qquad (\deltab,\xib,\phib)\mapsto T(\deltab,\xib,\phib):=L_{\deltab,\xib}(\phib)-R_{\deltab,\xib}-N_{\deltab,\xib}(\phib).
\]
Take $\phib^{\deltab,\xib}$, so that $T(\deltab,\xib,\phib^{\deltab,\xib})=0$. Then
\[
D_{\phib}T(\deltab,\xib,\phib^{\deltab,\xib})[\psib]=L_{\deltab,\xib}(\psib)-N_{\deltab,\xib}'(\phi^{\deltab,\xib})[\psib],\qquad \forall \psib\in K_{\deltab,\xib}^\perp.
\]
We will show that this map is injective. Let $\psib\in K_{\deltab,\xib}^\perp$ be such that
\[
L_{\deltab,\xib}(\psib)=N_{\deltab,\xib}'(\phib^{\deltab,\xib})[\psib].
\]
On the one hand, recall from Lemma \ref{lemma:L_invertible} that
\begin{equation}\label{eq:comparing1}
\|\psib\|_{H^1_0}\leq c \|L_{\deltab,\xib}(\psib)\|_{H^1_0}
\end{equation}
for some $c>0$. On the other hand, recalling the expression of $N_{\deltab,\xib}(\phib)$ from \eqref{eq:N}, we have
\begin{multline*}
N_{\deltab,\xib}'(\phib^{\deltab,\xib})[\psib]=\Pi_i^\perp\circ i^* \left[ \mu_i \psi_i \left( f'(\mu_i^{-1/2}PU_i+\phib_i^{\deltab,\xib})-f'(\mu_i^{-1/2}PU_i)  \right) + 2\beta \mu_i^{-1/2} (PU_i) \sum_{j\neq i} \phi_j^{\deltab,\xib}\psi_j  \right.\\
			\left. +\beta \psi_i \sum_{j\neq i}  (\phi_j^{\deltab,\xib})^2+2\beta  \phi_i^{\deltab,\xib}\sum_{j\neq i}  \phi_j^{\deltab,\xib}\psi_j +2\beta \psi_i \sum_{j\neq i}\mu_j^{-1/2} (PU_j)  \phi_j^{\deltab,\xib} + 2\beta  \phi_i^{\deltab,\xib} \sum_{j\neq i} \mu_j^{-1/2}(PU_j)\psi_j	\right].
\end{multline*}
Thus there exist $C_1,C_2>0$ such that
\begin{multline}\label{eq:comparing2}
\|N_{\deltab,\xib}'(\phib^{\deltab,\xib})[\psib]\|_{H^1_0} \leq  C_1\sum_{j\neq i} \left( \|(PU_i)\phi_i^{\deltab,\xib}\psi_i \|_{4/3} +\|(\phi_i^{\deltab,\xib})^2\psi_i\|_{4/3} +|\beta| \|(PU_i) \phi_j^{\deltab,\xib} \psi_j\|_{4/3} +\right. \\
			 +\left.  |\beta| \|\psi_i (\phi_j^{\deltab,\xib})^2\|_{4/3} +|\beta| \| \phi_i^{\deltab,\xib} \phi_j^{\deltab,\xib} \psi_j \|_{4/3} +|\beta| \|(PU_j) (\psi_i \phi_j^{\deltab,\xib}+\phi_i^{\deltab,\xib}\psi_j) \|_{4/3}\right)\\
			\leq  C_2 \left(\|\phi^{\deltab,\xib}\|_{H^1_0} \|\psib\|_{H^1_0} + \|\phi^{\deltab,\xib}\|_{H^1_0}^2 \|\psib\|_{H^1_0} +|\beta | \|\phib^{\deltab,\xib}\|_{H^1_0} \|\psib\|_{H^1_0}  \right.\\
			+\left. |\beta| \|\phi^{\deltab,\xib}\|_{H^1_0}^2 \|\psib\|_{H^1_0} +|\beta| \|\phib^{\deltab,\xib}\|_{H^1_0} \|\psib\|_{H^1_0}\right)\\
			\leq  (1+|\beta|)\ob_{\deltab}(1) \|\psib\|_{H^1_0}
\end{multline}
as $\deltab\to 0$, since $\|\phi^{\deltab,\xib}\|_{H^1_0}\to 0 $ as $\deltab\to 0$ (see Proposition \ref{prop:estimate_error}). Comparing \eqref{eq:comparing1} and \eqref{eq:comparing2}, we conclude that $\psib=0$ for sufficiently small $\eps_0>0$, which proves the lemma. \qedhere
\end{proof}

Given $\eta>0$, take $\eps_0>0$ as before, and define the \emph{reduced} $C^1$--functional $\widetilde E:(0,\eps_0)^m\times X_\eta\to \R$ by
\[
\widetilde E(\deltab,\xib)=E(PU_{\delta_1,\xi_1}+\phi_1^{\deltab,\xib},\ldots, PU_{\delta_m,\xi_m}+\phi_m^{\deltab,\xib}).
\]

\begin{proposition}\label{prop:Reduction}
Under the previous notations, we have that 
\[
(\mu_1^{-1/2}PU_{\delta_1,\xi_1}+\phi_1^{\deltab,\xib},\ldots, \mu_m^{-1/2}PU_{\delta_m,\xi_m}+\phi_m^{\deltab,\xib}) \text{ is a critical point of $E$}
\] 
 if, and only if, 
 \[
 (\deltab,\xib)\text{  is a critical point of $\widetilde E$. }
 \]
\end{proposition}
\begin{proof}
From the definition of $\phib^{\deltab,\xib}$, we know that
\[
E'(\mu_1^{-1/2}PU_{\delta_1,\xi_1}+\phi_1^{\deltab,\xib},\ldots,\mu_m^{-1/2}PU_{\delta_m,\xi_m}+\phi_m^{\deltab,\xib})(\varphi_1,\ldots, \varphi_m)=0\quad \forall (\varphi_1,\ldots, \varphi_m)\in K_1^\perp\times\ldots\times K_m^\perp.
\]
To simplify notations, we let $V_i^{\deltab,\xib}:=\mu_i^{-1/2}P U_{\delta_i,\xi_i}+\phi_i^{\deltab,\xib}$. For each $i=1,\ldots,4$, write $\xi_i=((\xi_i)_1,\dots,(\xi_i)_4)$. The claim follows from the fact that
\begin{align*}
\partial_{\delta_1} \widetilde E (\deltab,\xib)&=E'(V_1^{\deltab,\xib},\ldots, V_m^{\deltab,\xib})[ \mu_1^{-1/2}P \partial_{\delta_1} U_{\delta_1,\xi_1}+\partial_{\delta_1}  \phi^{\deltab,\xib}_1,\partial_{\delta_2}\phi_2,\ldots, \partial_{\delta_1}\phi^{\deltab,\xib}_m]   \\
			&=E'(V_1^{\deltab,\xib},\ldots, V_m^{\deltab,\xib})[\mu_1^{-1/2}P \partial_{\delta_1} U_{\delta_1,\xi_1},0\ldots, 0]\\
			(\ldots)\\
\partial_{\delta_m} \widetilde E (\deltab,\xib)&=E'(V_1^{\deltab,\xib},\ldots, V_m^{\deltab,\xib})[\partial_{\delta_m}  \phi_1^{\deltab,\xib},\partial_{\delta_m}\phi_2^{\deltab,\xib},\ldots, \mu_m^{-1/2}P \partial_{\delta_m} U_{\delta_m,\xi_m}+\partial_{\delta_m}\phi_m^{\deltab,\xib}]   \\
			&=E'(V_1^{\deltab,\xib},\ldots, V_m^{\deltab,\xib})[0,\ldots, 0, \mu_m^{-1/2} P \partial_{\delta_m} U_{\delta_m,\xi_m}]\\			   
\end{align*}
and similarly for $\partial_{({\xi_i})_j} \widetilde E(\deltab,\xib)$, $i=1,\ldots, m$, $j=1,\ldots, 4$. Thus the conclusion follows by recalling that, for each $i=1,\ldots, m$, the space $K_{\delta_i,\xi_i}$ is spanned by $\{P\partial_{\delta_i} U_{\delta_i,\xi_i}, P\partial_{({\xi_i})_1}U_{\delta_i,\xi_i}, \ldots  P\partial_{({\xi_i})_4}U_{\delta_i,\xi_i}\}$.
\end{proof}

\section{Expansion of the Reduced Energy}\label{sec:Energy_expansion}\label{sec:EnergyExpansion}

In the following, we take a small $\eta>0$, and $\xib\in X_\eta$, $\deltab\in (0,\eps_0)^m$, and denote $\phib:=\phib^{\deltab,\xib}$ and $U_i:= U_{\delta_i,\xi_i}$. A direct computation shows that
\[
\widetilde E(\deltab,\xib) = E(\mu_1^{-1/2}PU_1,\ldots, \mu_m^{-1/2}PU_m)+R(\deltab,\xib)
\]
where
\begin{equation}
\begin{split}\label{eq:R(delta,xi)}
R(\deltab,\xib)=&\sum_{i=1}^m \frac{1}{2} \int_\Omega (2\mu_i^{-1/2} U_i^3\phi_i+|\nabla \phi_i|^2)
						-\sum_{i=1}^m \mu_i\int_\Omega (F(\mu_i^{-1/2} PU_i+\phi_i)-F(\mu_i^{-1/2}PU_i))\\
& -\sum_{i=1}^m \frac{\lambda_i}{2} \int_\Omega (2\mu_i^{-1/2}(PU_i)\phi_i+\phi_i^2) -\frac{\beta}{2}\sum_{i<j} \int_\Omega \left(2\mu_i^{-1}\mu_j^{-1/2}(PU_i)^2(PU_j)\phi_j + \mu_i^{-1}(PU_i)^2\phi_j^2 \right.\\
&+2\mu_i^{-1/2}\mu_j^{-1}(PU_i)\phi_i (PU_j)^2 + 4 \mu_i^{-1/2}\mu_j^{-1/2}(PU_i)\phi_i (PU_j)\phi_j + 2\mu_i^{-1/2}(PU_i)\phi_i \phi_j^2 \\
&\left. + \mu_j^{-1}\phi_i^2 (PU_j)^2+2\mu_j^{-1/2}\phi_i^2(PU_j)\phi_j+\phi_i^2\phi_j^2 \right)
\end{split}
\end{equation}
We have $E(\mu_1^{-1/2}PU_1,\ldots, \mu_m^{-1/2}PU_m)=\sum_{i=1}^m (A_i-B_i-C_i)-\sum_{i<j} D_{ij}$, where
\[
A_i=\frac{1}{2}\int_\Omega \mu_i^{-1}|\nabla PU_i|^2,\quad B_i=\mu_i \int_\Omega F(\mu_i^{-1/2}PU_i),\quad C_i=\frac{\lambda_i}{2}\int_\Omega \mu_i^{-1}(PU_i)^2
\]
and
\[
D_{ij}=\frac{\beta}{2}\int_\Omega \mu_i^{-1}\mu_j^{-1}(PU_i)^2(PU_j)^2.
\]
We also denote
\[
A=\int_{\R^4} U_{1,0}^3,\qquad B=\int_{\R^4} U_{1,0}^4.
\]
Integrating by parts, using the equation $-\Delta PU_i=U_i^3$, the expansion of Lemma \ref{lemma:expansion_of_PU}, and the integral estimates of Lemma \ref{lemma:integral_estimates}, we have that
\[
\begin{split}
A_i&=\frac{1}{2}\int_\Omega \mu_i^{-1}U_i^3(PU_i)=\frac{1}{2}\int_\Omega \mu_i^{-1}U_i^3 (U_i-A \delta_i H(x,\xi_i)+\textrm{o}(\delta_i))\, dx\\
	&=\frac{1}{2}\int_\Omega \mu_i^{-1} U_i^4 -\frac{A \delta_i}{2}\mu_i^{-1}\int_\Omega U_i^3 H(x,\xi_i)\, dx+\textrm{o}(\delta_i)\int_\Omega U_i^3\\
	&=\frac{1}{2}c_4^4 \mu_i^{-1}\int_\Omega \frac{ \delta_i^4}{(\delta_i^2+|x-\xi_i|^2)^4}-\frac{A \delta_i}{2}c_4^3\mu_i^{-1}\int_\Omega \frac{\delta_i^3}{(\delta_i^2+|x-\xi_i|^2)^3}H(x,\xi_i)+\textrm{o}(\delta_i^2)\\
	&= \frac{1}{2}c_4^2\mu_i^{-1}\int_\frac{\Omega-\xi_i}{\delta_i} \frac{1}{(1+|y|^2)^4}-\frac{A \delta_i^2}{2}c_4^3\mu_i^{-1}\int_{\frac{\Omega-\xi_i}{\delta_i}} \frac{1}{(1+|y|^2)^3}H(\xi_i+\delta_i y,\xi_i) +\ob(\delta_i^2)\\
	&=\frac{c_4^4}{2}\mu_i^{-1}B-\frac{c_4^3}{2}\mu_i^{-1} A^2 \tau(\xi_i)\delta_i^2 +\textrm{o}(\delta_i^2)\qquad \text{ as } \delta_i \to 0,
\end{split}
\]
recalling that $\tau$ denotes the Robin function, i.e., $\tau(x)=H(x,x)$.
As for $B_i$, we have
\[
\begin{split}
B_i&= \frac{1}{4}\int_\Omega \mu_i^{-1} (PU_i)^4=\frac{1}{4}\int_\Omega \mu_i^{-1} (U_i-\delta_i A H(x,\xi_i)+\textrm{o}(\delta_i) )^4 \\
	&= \frac{1}{4} \mu_i^{-1} \int_\Omega (U_i-\delta_i A H(x,\xi_i))^4 +\ob(\delta_i^2)\\
	&= \frac{1}{4}\mu_i^{-1} \int_\Omega U_i^4 -\mu_i^{-1} \delta_i A \int_\Omega U_i^3 H(x,\xi_i)+\ob(\delta_i^2)\\
	&=\frac{c_4^4}{4}\mu_i^{-1} \int_\Omega \frac{\delta_i^4}{(\delta_i^2+|x-\xi_i|^2)^4} - c_4^3  \mu_i^{-1} A \delta_i  \int_\Omega \frac{\delta_i^3}{(\delta_i^2+|x-\xi_i|^2)^3}H(x,\xi_i)+\ob(\delta_i^2)\\
	&=\frac{c_4^4}{4}\mu_i^{-1} \int_{\frac{\Omega-\xi_i}{\delta_i}} \frac{1}{(1+|y|^2)^4}-c_4^3 \mu_i^{-1}A  \delta_i^2\int_{\frac{\Omega-\xi_i}{\delta_i}} \frac{1}{(1+|y|^2)^3}H(\xi_i+\delta_i y,\xi_i) + \ob(\delta_i^2)\\
	&= \frac{c_4^4}{4}\mu_i^{-1} B-c_4^3 \mu_i^{-1}A^2 \tau(\xi_i) \delta_i^2 +\ob(\delta_i^2)\qquad \text{ as } \delta_i\to 0.
\end{split}
\]
Moreover,
\[
\begin{split}
C_i&=\frac{\lambda_i}{2}\int_\Omega \mu_{i}^{-1} (PU_i)^2=\frac{\lambda_i \mu_i^{-1}}{2}\int_\Omega (U_i-A \delta_i H(x,\xi_i) +\ob(\delta_i))^2\\
	&= \frac{\lambda_i \mu_i^{-1}}{2}\int_\Omega (U_i-A \delta_i H(x,\xi_i))^2+\lambda_i\Ob(\delta_i^2)\\
	&= \frac{\lambda_i \mu_i^{-1}}{2}\int_\Omega U_i^2 - \lambda_i \mu_i^{-1} \delta_i A \int_\Omega U_i H(x,\xi_i) + \lambda_i \Ob(\delta_i^2)\\
	&=\frac{c_4^2}{2}\mu_i^{-1}\omega_3 \lambda_i \delta_i^2 |\ln \delta_i| + \lambda_i \textrm{O}(\delta_i^2) \qquad \text{ as } \delta_i \to 0.
\end{split}
\]
Finally, for some constant $C>0$,
\[
\begin{split}
D_{ij} &=\frac{\beta}{2}\mu_i^{-1}\mu_j^{-1}\int_\Omega (PU_i)^2(PU_j)^2\\
	&=\frac{\beta}{2}\mu_i^{-1}\mu_j^{-1} \int_\Omega (U_i-\delta_i A H(x,\xi_i)+\ob(\delta_i))^2(U_j-\delta_j A H(x,\xi_j)+\ob(\delta_j))^2\\
	&= \frac{\beta}{2}\mu_i^{-1}\mu_j^{-1} \int_\Omega (U_i^2+\Ob(\delta_i) U_i +\Ob(\delta_i^2)) (U_j^2+\Ob(\delta_j) U_j +\Ob(\delta_j^2))\\
	&=\frac{\beta}{2}\mu_i^{-1}\mu_j^{-1}\int_\Omega U_i^2 U_j^2 + \beta \textrm{O}(\delta_i^2\delta_j^2|\ln \delta_i\delta_j|).
\end{split}
\]
Moroever, since
\[
\begin{split}
\int_\Omega U_i^2 U_j^2=&\int_{B_{\eta/2}(\xi_i)} U_i^2 U_j^2+\int_{B_{\eta/2}(\xi_j)} U_i^2 U_j^2 + \Ob(\delta_i^2 \delta_j^2)\\
					=&c_4^4 \int_{B_{\eta/(2\delta_i)}(0)} \frac{\delta_i^2 \delta_j^2}{(1+|y|^2)^2(\delta_j^2+|\delta_i y+\xi_i-\xi_j|^2)^2}\\
					&+c_4^4 \int_{B_{\eta/(2\delta_j)}(0)} \frac{\delta_i^2 \delta_j^2}{(1+|y|^2)^2(\delta_i^2+|\delta_j y+\xi_j-\xi_i|^2)^2}+\Ob(\delta_i^2\delta_j^2)\\
					\leq & c_4^4 \kappa \delta_i^2 \delta_j^2 \left(\int_0^{\eta/(2\delta_i)} \frac{t^3}{(1+t^2)^2}+\int_0^{\eta/(2\delta_j)} \frac{t^3}{(1+t^2)^2}\right)+\Ob(\delta_i^2\delta_j^2) \\
					=&\Ob(\delta_i^2 \delta_j^2 |\ln \delta_i \delta_j|)\qquad \text{ as } (\delta_i,\delta_j)\to (0,0),
\end{split}
\] we have
\[
D_{ij}=\Ob(\delta_i^2 \delta_j^2 |\ln \delta_i \delta_j|).
\]

In conclusion, we have deduced the following asymptotic expansion:
\begin{multline*}
E(\mu_1^{-1/2}PU_1,\ldots, \mu_m^{-1/2}PU_m)=\sum_{i=1}^m (A_i-B_i-C_i)-\sum_{i<j} D_{ij}\\
= \sum_{i=1}^m \frac{c_4^4}{4}\mu_i^{-1} B + \sum_{i=1}^m \left(\frac{c_4^3}{2}\mu_i^{-1} A^2 \tau(\xi_i) \delta_i^2-\frac{c_4^2}{2} \mu_i^{-1} \omega_3 \lambda_i \delta_i^2 |\ln \delta_i| \right) \\
		+\sum_{i=1}^m (\ob(\delta_i^2)+ \lambda_i \Ob(\delta_i^2))+\sum_{i<j} |\beta| \Ob(\delta_i^2 \delta_j^2|\ln \delta_i \delta_j|).
\end{multline*}
Let us now estimate $R(\deltab,\xib)$, given by expression \eqref{eq:R(delta,xi)}. We recall from Proposition \ref{prop:estimate_error} that
\[
\|\phib^{\deltab,\xib}\|_{H^1_0(\Omega)} \leq c \sum_{i=1}^m\left( \lambda_i \delta_i +\delta_i^2  \right).
\]
Since we consider $\beta<0$ fixed, we drop (as in the previous asymptotic formula) the dependence of the quantities from $\beta$. All the constants appearing from now on will depend on this parameter.
\begin{lemma}\label{lemma:remainderestimate}Give $\eta>0$ small, we have
\[
|R(\deltab,\xib)| \leq C \sum_{i=1}^m\left( \lambda_i \delta_i^3+ \delta_i^4  \right).
\]
as $\deltab\to 0$, uniformly for every $\xib\in K_{\eta}$, .
\end{lemma}
\begin{proof}
We have
\begin{multline*}
\sum_{i=1}^m \frac{1}{2}\int_\Omega 2 \mu_i^{-1/2} U_i^3 \phi_i - \sum_{i=1}^m \mu_i \int_\Omega (F(\mu_i^{-1/2}PU_i+\phi_i)-F(\mu_i^{-1/2} PU_i))\\
= -\sum_{i=1}^m \mu_i \int_\Omega (F(\mu_i^{-1/2} PU_i+\phi_i)-F(\mu_i^{-1/2}PU_i)-F'(\mu_i^{-1/2} PU_i)\phi_i)+\sum_{i=1}^m \int_\Omega \mu_i^{-1/2} (U_i^3-(PU_i)^3)\phi_i
\end{multline*}
As for the first term, by using Lemma \ref{lemma:auxiliaryestimate} we have, for $\deltab$ small,
\[
\begin{split}
\left| -\sum_{i=1}^m \mu_i \int_\Omega \right.&\left.(F(\mu_i^{-1/2} PU_i+\phi_i)-F(\mu_i^{-1/2}PU_i)-F'(\mu_i^{-1/2} PU_i)\phi_i)  \right|\\
									&\leq C \sum_{i=1}^m \int_\Omega ((PU_i)^2\phi_i^2 +\phi_i^4) \leq C\sum_{i=1}^m (\| PU_i\|_4^2 \|\phi_i\|_4^2  +\|\phi_i\|_4^4)\\
									& \leq C\| \phib\|_{H^1_0}^2 \leq C' \sum_{i=1}^m\left( \lambda_i^2 \delta_i^2 +\delta_i^4  \right).
\end{split}
\]
As for the other term, using Lemmas \ref{lemma:expansion_of_PU} and \ref{lemma:auxiliaryestimate}:
\[
\begin{split}
\left|\sum_{i=1}^m \int_\Omega \mu_i^{-1/2} (U_i^3-(PU_i)^3)\phi_i \right| &\leq C\sum_{i=1}^m \int_				\Omega \left|U_i^2 \Ob(\delta_i)+\Ob(\delta_i^3) \right| \left|\phi_i\right|\\
		&\leq \sum_{i=1}^m\Ob(\delta_i) \|\phi_i\|_4 \|U_i\|_{8/3}^2 +  \|\phib\|_{H^1_0} \sum_{i=1}^m \Ob(\delta_i^3)\\
		&\leq C \sum_{i=1}^m \left( \lambda_i \delta_i^3+\delta_i^4\right)
\end{split}
\]
Moreover,
\[
\left| \sum_{i=1}^m \frac{1}{2}\int_\Omega |\nabla \phi_i|^2-\sum_{i=1}^m \frac{\lambda_i}{2}\int_\Omega \phi_i^2\right| \leq C  \|\phib\|_{H^1_0}^2 \leq C' \sum_{i=1}^m\left( \lambda_i^2 \delta_i^2 +\delta_i^4  \right),
\]
\begin{align*}
\left|-\sum_{i=1}^m \frac{\lambda_i}{2}\int_\Omega 2\mu_i^{-1/2}(PU_i)\phi_i \right| &\leq \sum_{i=1}^m \frac{\lambda_i}{2} \int_\Omega  (U_i+\Ob(\delta_i))|\phi_i| \leq \sum_{i=1}^m \lambda_i \Ob(\delta_i) \|\phib\|_{H^1_0}\\
&\leq C\sum_{i=1}^m (\lambda_i^2\delta_i^2 + \lambda_i \delta_i^3)
\end{align*}
and
\[
\begin{split}
\left| -\frac{\beta}{2}\sum_{i<j} \int_\Omega  \left( \mu_i^{-1}(PU_i)^2\phi_j^2  + 4 \mu_i^{-1/2}\mu_j^{-1/2}(PU_i)\phi_i (PU_j)\phi_j + 2\mu_i^{-1/2}(PU_i)\phi_i \phi_j^2 \right.\right.\\
	\left.\left. + \mu_j^{-1}\phi_i^2 (PU_j)^2+2\mu_j^{-1/2}\phi_i^2(PU_j)\phi_j+\phi_i^2\phi_j^2 \right)\right|\\
\leq \frac{|\beta|}{2} \sum_{i<j} \left( \|PU_i\|_4^2 \|\phi_j\|_4^2 + \|PU_i\|_4 \|PU_j\|_4  \|\phi_i\|_4 \|\phi_j\|_4 + \|PU_i\|_4 \|\phi_i\|_4 \|\phi_j\|^2_4\right.\\
	\left. +\|\phi_i\|_4^2 \| PU_j\|_4^2  +\|\phi_i\|_4^2 \|PU_j\|_4 \|\phi_j\|_4 +\|\phi_i\|_4^2 \|\phi_j\|_4^2\right)\\
\leq C \sum_{i<j} |\beta|\left(  \|\phib\|_{H^1_0}^2 +\|\phib\|_{H^1_0}^3 +\|\phib\|_{H^1_0}^4\right) \leq C' |\beta| \| \phib\|_{H^1_0}^2 \leq C'' |\beta| \sum_{i=1}^m\left( \lambda_i^2 \delta_i^2 +\delta_i^4  \right).
\end{split}
\]
Finally, using Lemma \ref{lemma:expansion_of_PU}--\ref{lemma:upper_estimate_auxiliar},
\[
\begin{split}
\left| -\frac{\beta}{2}\sum_{i<j} \right.& \left.\int_\Omega \left(2\mu_i^{-1}\mu_j^{-1/2}(PU_i)^2(PU_j)\phi_j +2\mu_i^{-1/2}\mu_j^{-1}(PU_i)\phi_i (PU_j)^2 \right) \right| \\
		& \leq  C|\beta| \sum_{i\neq j} \int_\Omega \left| (PU_i)^2(PU_j)\phi_j \right| \leq C |\beta| \|\phib\|_{H^1_0} \left(\int_\Omega (PU_i)^{8/3}(PU_j)^{4/3}\right)^{3/4} \\
		& =  |\beta|\Ob(\delta_i\delta_j) \|\phib\|_{H^1_0}\leq C'|\beta| \sum_{i=1}^m (\lambda_i \delta_i^4+\delta_i^4).
\end{split}
\]
The proof is then complete, combining all the previous estimates.
\end{proof}

In conclusion, we have the following energy expansion:
\begin{multline}\label{eq:energyexpansion_lambda}
\widetilde E(\deltab,\xib)= \sum_{i=1}^m \frac{c_4^4}{4}\mu_i^{-1} B + \sum_{i=1}^m \left(\frac{c_4^3}{2}\mu_i^{-1} A^2 \tau(\xi_i) \delta_i^2-\frac{c_4^2}{2} \mu_i^{-1} \omega_3 \lambda_i \delta_i^2 |\ln \delta_i| \right) \\
		+\sum_{i=1}^m (\ob(\delta_i^2)+ \lambda_i \Ob(\delta_i^2))+\sum_{i<j} |\beta| \Ob(\delta_i^2 \delta_j^2|\ln \delta_i \delta_j|),
\end{multline}
uniformly for every $\lambda_i$ small, $\xi$ in a compact set of $X$, as $\delta_i\to 0$.

\section{Proofs of the main results}\label{sec:Proofs}

In this section we prove our main results, namely Theorems \ref{thm:main1} and \ref{thm:main2}.

\subsection{Proof of Theorem \ref{thm:main1}}

Choose
\[
\delta_i=e^{-\frac{d_i}{\lambda_i}}, \qquad \text{ for some } d_i>0.
\]
Then, recalling that $c_4=2\sqrt{2}$, we have
\begin{multline*}
\sum_{i=1}^m \left(\frac{c_4^3}{2} \mu_i^{-1}A^2 \tau(\xi_i) \delta_i^2-\frac{c_4^2}{2} \mu_i^{-1} \omega_3 \lambda_i \delta_i^2 |\ln \delta_i| \right)\\
		= \sum_{i=1}^m e^{-\frac{2d_i}{\lambda_i}} \left( 8\sqrt{2} \mu_i^{-1}  A^2 \tau (\xi_i)-4\mu_i^{-1}\omega_3 d_i\right)=:\sum_{i=1}^m \Psi_{\lambda_i}(d_i,\xi_i)=: \Psi_{\lambdab}(\mathbf{d},\xib),
\end{multline*}
and thus
\begin{equation}\label{eq:expansionfinal}
\widetilde E(\deltab,\xib)=\widetilde E(\mathbf{d},\xib)=\sum_{i=1}^m 16\mu_i^{-1} B + \Psi_{\mathbf{\lambda}}(\mathbf{d},\xib)+\sum_{i=1}^m \ob(e^{-\frac{2d_i}{\lambda_i}})
\end{equation}
as $\eps:=\max\{\lambda_1,\ldots, \lambda_m\}\to 0^+$, uniformly in every compact set of $[0,+\infty)^m \times X$.

\begin{lemma}
Assume that the Robin function $\tau$ achieves $m$ distinct local minimums $\bar \xi_1,\ldots, \bar \xi_m$. Then, for each $\lambdab=(\lambda_1,\ldots, \lambda_m)$, the function $\Psi_{\lambdab}$ achieves a local minimum at the point
\[
(\mathbf{d},\xib)=\left(\left(\frac{\lambda_1}{2}+\frac{2\sqrt{2}A^2}{\omega_3}\tau(\bar \xi_1),\ldots, \frac{\lambda_m}{2}+\frac{2\sqrt{2}A^2}{\omega_3}\tau(\bar \xi_m)\right) , (\bar \xi_1,\ldots, \bar \xi_m)\right).
\]
\end{lemma}
\begin{proof}
Take mutually disjoint open sets $\Lambda_1,\ldots, \Lambda_m$ such that, for each $i$, there exists $\bar \xi_i\in \Lambda_i$ such that $\min_{\overline{\Lambda}_i} \tau =\tau(\bar \xi_i)$. Then $\Psi_{\lambda_i} (d_i,\xi_i)\geq \Psi_{\lambda_i}(d_i,\bar \xi_i)$ for every $d_i\geq 0$, $\xi_i\in \overline{\Lambda}_i$. Since $\lim_{d_i\to 0} \Psi_{\lambda_i}(d_i,\bar \xi_i)=8\sqrt{2} A^2 \mu_i^{-1} \tau(\bar \xi_i)>0$, and $\lim_{d_i\to +\infty} \Psi_{\lambda_i} (d_i,\bar \xi_i)=0^-$, then in fact $\Psi_{\lambda_i}$ achieves an interior minimizer in $[0,+\infty) \times \overline{\Lambda}_i$. We can compute such minimiser directly:
\[
\frac{\partial}{\partial d_i} \Psi_{\lambda_i}(d_i,\bar \xi_i)=0 \iff d_i=\frac{\lambda_i}{2}+\frac{2\sqrt{2}A^2}{\omega_3}\tau(\bar \xi_i). \qedhere
\]
\end{proof}

\begin{proof}[Conclusion of the proof of Theorem \ref{thm:main1}] a) We can write:
\[
\widetilde E(\mathbf{d},\xib)=\sum_{i=1}^m 16 \mu_i^{-1} B+\sum_{i=1}^m e^{-\frac{2d_i}{\lambda_i}} \left( 8\sqrt{2} A^2 \mu_i^{-1}\tau (\xi_i)-4\mu_i^{-1}\omega_3 d_i  +\ob(1)\right),
\]
where $\ob(1)\to 1$ as $\eps=\max\{\lambda_1,\ldots, \lambda_m\}\to 0$. 

Fix $\gamma>0$ small. We claim that the function
\[
\varphi_{\lambdab}(d_i,\xi_i):=e^{-\frac{2d_i}{\lambda_i}} \left( 8\sqrt{2} A^2 \mu_i^{-1}\tau (\xi_i)-4\mu_i^{-1}\omega_3 d_i  +\ob(1)\right)
\]
admits a minimum in the interior of the set $S_{i,\gamma}:= [\frac{2\sqrt{2}}{\omega_3}A^2 \tau (\bar \xi_i)-\gamma,\frac{2\sqrt{2}}{\omega_3}A^2 \tau (\bar \xi_i)+\gamma] \times \Lambda_i$, if $\eps=\max\{\lambda_i\}$ is sufficiently small. In fact:
\begin{itemize}
\item[-] for $(d_i,\xi_i)\in   [\frac{2\sqrt{2}}{\omega_3}A^2 \tau (\bar \xi_i)-\gamma,\frac{2\sqrt{2}}{\omega_3}A^2 \tau (\bar \xi_i)+\gamma] \times \partial \Lambda_i$, one has $-d_i\geq -\frac{2\sqrt{2}}{\omega_3}A^2 \tau(\bar \xi_i)-\gamma$, so
\[
\begin{split}
\varphi_{\lambdab}(d_i,\xi_i)&\geq e^{-\frac{4\sqrt{2}}{\lambda_i \omega_3}A^2 \tau(\bar \xi_i)-\frac{2\gamma}{\lambda_i}}(8\sqrt{2}A^2 \mu_i^{-1}(\tau(\xi_i)-\tau(\bar \xi_i))-4\mu_i^{-1}\omega_3 \gamma +\ob(1))\\
						&\geq e^{-\frac{4\sqrt{2}}{\lambda_i \omega_3}A^2 \tau(\bar \xi_i)-\frac{2\gamma}{\lambda_i}}(-4\mu_i^{-1}\omega_3 \gamma +\ob(1)).
\end{split}
\]
\item[-] for $(d_i,\xi_i)\in   \{\frac{2\sqrt{2}}{\omega_3}A^2 \tau (\bar \xi_i)-\gamma\} \times \Lambda_i$,
\[
\varphi_{\lambdab}(d_i,\xi_i)\geq e^{-\frac{4\sqrt{2}}{\lambda_i\omega_3}A^2\tau(\bar \xi_i)+\frac{2\gamma}{\lambda_i}}(4\mu_i^{-1}\omega_3 \gamma +\ob(1))>0,
\]
\item[-] for $(d_i,\xi_i)\in \{\frac{2\sqrt{2}}{\omega_3}A^2 \tau (\bar \xi_i)+\gamma\} \times \Lambda_i$,
\[
\varphi_{\lambdab}(d_i,\xi_i)\geq e^{-\frac{4\sqrt{2}}{\lambda_i\omega_3}A^2 \tau(\bar \xi_i)-\frac{2\gamma}{\lambda_i}}(-4\mu_i^{-1}\omega_3 \gamma +\ob(1)).
\]
\end{itemize}
Thus we have
\[
\min_{\partial S_{i,\gamma}}\varphi_{\lambdab} \geq e^{-\frac{4\sqrt{2}}{\lambda_i\omega_3}A^2-\frac{2\gamma}{\lambda_i}}(-4\mu_i^{-1}\omega_3 \gamma +\ob(1))
\]
On the other hand, evaluating $\varphi_{\lambdab}$ at $(\frac{2\sqrt{2}}{\omega_3}A^2 \tau (\bar \xi_i)+\frac{\gamma}{2}, \bar \xi_i)$ (which lies in the interior of $S_{i,\gamma}$), we have
\[
\begin{split}
\varphi_{\lambdab} &=e^{-\frac{4\sqrt{2}}{\lambda_i\omega_3}A^2 \tau (\bar \xi_i)-\frac{\gamma}{\lambda_i}}(-2\mu_i^{-1}\omega_3 \gamma +\ob(1)).   \\
\end{split}
\]
Since
\[
e^{-\frac{4\sqrt{2}}{\lambda_i\omega_3}A^2-\frac{\gamma}{\lambda_i}}(-2\mu_i^{-1}\omega_3 \gamma +\ob(1))	< e^{-\frac{4\sqrt{2}}{\lambda_i\omega_3}A^2-\frac{2\gamma}{\lambda_i}}(-4\mu_i^{-1}\omega_3 \gamma +\ob(1))
\]
the claim is proved.

Thus, there exists $d_{i,\lambdab}\to \frac{2\sqrt{2}}{\omega_3}A^2 \tau(\bar \xi_i)$ as $\lambdab\to 0$, and points $\xi_i^{\lambdab}$ converging to a local minimum of $\tau$ in $\Lambda_i$, such that, defining $\delta_i^{\lambdab}=e^{-\frac{d_i \lambda_i}{\lambda_i}}$, then $(\deltab^{\lambdab},\xib^{\lambdab})$ is a minimiser of $\widetilde E$, hence $\widetilde E'(\delta^\lambda,\xib^\lambda)=0$. By invoking Proposition \ref{prop:Reduction}, we have obtained a solution of \eqref{eq:system_2eq_2} of the form
\[
u_i^{\lambdab}=\mu_i^{-1/2}PU_{\delta_{i}^{\lambdab},\xi_{i}^{\lambdab}}+\phi_{i}^{\lambdab}.
\]

\medbreak

\noindent b) Now, to conclude, we just need to show that  $u_i^{\lambdab}$ is positive:
\begin{itemize}
\item If $\beta\leq 0$, by multiplying the $i$--th equation in \eqref{eq:system_2eq_2} by $u_i^-$, we obtain
\[
-\int_\Omega |\nabla u_i^-|^2+\lambda_i \int_\Omega (u_i^-)^2=-\beta \sum_{j\neq i}\int_\Omega (u_i^-)^2  u_j^2\geq 0.
\]
and so, if  $u_i^-\not\equiv 0$, then
\[
0\leq -\int_\Omega |\nabla u_i^-|^2+\lambda_i \int_\Omega (u_i^-)^2\leq \int_\Omega |\nabla u_i^-|^2\left(-1+\frac{\lambda_i}{\lambda_1(\Omega)}\right)<0,
\]
a contradiction.
\item If $0<\beta\leq \bar \beta$, observe that, going through the proof of Proposition \ref{prop:estimate_error}, we conclude that
\[
\|\phib^{\lambdab}\|_{H^1_0} \leq c \sum_{i=1}^m (\lambda_i \delta_i^{\lambdab} +(\delta_i^{\lambdab})^2),
\]
for $c$ only depending on $\beta$, and so $\|u_i^{\lambdab}\|_4\leq C$, for every $\lambdab\sim 0$. Thus, if we choose $\bar \beta$ sufficiently small, we cannot have $u_i^-\not\equiv 0$, otherwise by using Cauchy-Schwarz and Sobolev inequalities:
\[
\begin{split}
0 &=-\int_\Omega |\nabla u_i^-|^2+\lambda_i \int_\Omega (u_i^-)^2+\beta \sum_{j\neq i}\int_\Omega (u_i^-)^2  u_j^2 \\
    &\leq \int_\Omega |\nabla u_i^-|^2\left(-1+\frac{\lambda_i}{\lambda_1(\Omega)}\right) + \bar \beta \left(\int_\Omega (u_i^-)^4\right)^{1/2} \leq \int_\Omega |\nabla u_i^-|^2 \left(\frac{\lambda_i-\lambda_1(\Omega)}{2\lambda_1(\Omega)}\right)<0
\end{split}
\]
which is again a contradiction.
\end{itemize}
\end{proof}

\subsection{Proof of Theorem \ref{thm:main2}}

The proof of this theorem follows more or less the same lines of the previous one. The main fact is that we need to prove that the expansion \eqref{eq:energyexpansion_lambda} is $C^1$ with respect to the variables $\deltab$ and $\xib$.

\begin{lemma}\label{lemma:C^1energyexpansion}
We have, for every $i=1,\ldots, m$:
\begin{multline*}
\frac{\partial }{\partial \delta_i}\widetilde E(\deltab,\xib)=   c_4^3\mu_i^{-1} A^2 \tau(\xi_i) \delta_i - c_4^2 \mu_i^{-1} \omega_3 \lambda_i \delta_i |\ln \delta_i| + \frac{c_4^2}{2} \mu_i^{-1} \omega_3 \lambda_i \delta_i   \\
		+\sum_{i=1}^m (\ob(\delta_i)+ \lambda_i \Ob(\delta_i))+\sum_{i<j} |\beta| \Ob(\delta_i \delta_j^2|\ln \delta_i \delta_j|).
\end{multline*}
and, for each $i=1,\ldots, m$ and $k=1,\ldots, 4$:
\begin{multline*}
\frac{\partial }{\partial (\xi_i)_k} \widetilde E(\deltab,\xib)=  \sum_{i=1}^m \frac{c_4^3}{2}\mu_i^{-1} A^2 \frac{\partial \tau}{\partial (\xi_i)_k}(\xi_i) \delta_i^2+\sum_{i=1}^m (\ob(\delta_i^2)+ \lambda_i \Ob(\delta_i^2))+\sum_{i<j} |\beta| \Ob(\delta_i^2 \delta_j^2|\ln \delta_i \delta_j|),
\end{multline*}
uniformly for every $\lambda_i$ small, $\xi$ in a compact set of $X$, as $\delta_i\to 0$.
\end{lemma}

\begin{proof}
 Recalling the proof of Proposition \ref{prop:Reduction}, we have, for every $i=1,\ldots, m$:
\begin{equation}\label{eq:C1expansion_k=0}
\partial_{\delta_i} \widetilde E (\deltab,\xib)=E'(\mu_1^{-1/2}PU_1+\phi_1,\ldots, \mu_m^{-1/2}PU_m + \phi_m)[0,\ldots, 0,\mu_i^{-1/2}P \partial_{\delta_i} U_i,0\ldots, 0],
\end{equation}
and, for $i=1,\ldots, m$ and $k=1,\ldots, 4$,
\begin{equation}\label{eq:C1expansion_k>0}
\partial_{(\xi_i)_k} \widetilde E (\deltab,\xib) =E'(\mu_1^{-1/2}PU_1+\phi_1,\ldots, \mu_m^{-1/2}PU_m + \phi_m)[0,\ldots, 0,\mu_i^{-1/2}P\partial_{({\xi_i})_k}U_i,0\ldots, 0].
\end{equation}
Along this proof, we will use the notations $Z_i^0:= \partial_{\delta_i}U_i$ and $Z_i^k=\partial_{(\xi_i)_k} U_i$. Since the expansion in \eqref{eq:expansion1} is $C^1$, we have
\[
PZ_i^0=Z_i^0- A H(\cdot,\xi)+\ob(1),\qquad PZ_i^k=Z_i^k- \delta_i A  \frac{\partial H}{\partial (\xi_i)_k}(\cdot,\xi)+\ob(\delta_i)
\]
and the proof will be very similar to the energy expansion performed in Section \ref{sec:Energy_expansion}. For that reason, here we will present less computations.

\medbreak

We need to expand the quantity:
\begin{multline*}
 \int_\Omega (f(\mu_i^{1/2}U_i)-f(\mu_i^{-1/2} PU_i + \phi_i)) \mu_i^{1/2} (PZ_i^k) - \lambda_i \int_\Omega (\mu_i^{-1/2} PU_i + \phi_i)\mu_i^{-1/2} (PZ_i^k)\\
 		 -\beta  \int_\Omega (\mu_i^{-1/2} PU_i + \phi_i) \mu_i^{-1/2} (PZ_i^k) \sum_{j\neq i} (\mu_j^{-1/2} (PU_j)+\phi_j)^2
\end{multline*}
which coincides with \eqref{eq:C1expansion_k=0} for $k=0$, and with \eqref{eq:C1expansion_k>0} for $k\geq 1$. We can write this expression as
\begin{multline*}
\int_\Omega (f(\mu_i^{-1/2}U_i)-f(\mu_i^{-1/2}PU_i))\mu_i^{1/2} (PZ_i^k) - \lambda_i \mu_i^{-1}\int_\Omega (PU_i) (PZ_i^k) \\
-\beta \int_\Omega \mu_i^{-1} (PU_i) (PZ_i^k) \sum_{j\neq i} \mu_j^{-1} (PU_j)^2 +  R_i(\deltab,\xib),
\end{multline*}
with 
\begin{multline*}
R^k_i(\deltab,\xib)= \int_\Omega (f(\mu_i^{-1/2} PU_i)-f(\mu_i^{-1/2}PU_i + \phi_i))\mu_i^{1/2} (PZ_i^k) - \lambda_i \mu_i^{-1/2} \int_\Omega  \phi_i (PZ_i^k)\\
-\beta  \mathop{\sum_{j=1}^m}_{j\neq i}\int_\Omega (\mu_i^{-1} (PU_i) (PZ_i^k) \phi_j^2 + 2\mu_i^{-1}\mu_j^{-1/2}(PU_i) (PZ_i^k) (PU_j) \phi_j 	+ \mu_i^{-1/2}\mu_j^{-1} (PZ_i^k) \phi_i (PU_j)^2  \\
+ \mu_i^{-1/2} (PZ_i^k) \phi_i \phi_j^2 + 2\mu_i^{-1/2}\mu_j^{-1/2} (PZ_i^k) \phi_i (PU_j) \phi_j).
\end{multline*}

\medbreak

\noindent 1) For $k=0$, 
\[
\begin{split}
\int_\Omega (f(\mu_i^{-1/2}U_i)-f(\mu_i^{-1/2} PU_i + \phi_i)) \mu_i^{1/2} (PZ_i^0) &= \int_\Omega \mu_i^{-1} (U_i^3-(PU_i)^3)PZ_i^0\\
								&= 3\mu_i^{-1}\int_\Omega \delta_i U_i^2 AH(x,\xi_i) PZ_i^0 +\ob(\delta_i)\\
								&= 3\mu_i^{-1} c_4^3 A \delta_i^2 \tau(\xi_i) \int_{\R^4} \frac{|y|^2-1}{(1+|y|^2)^4} + \ob(\delta_i)\\
								&=\mu_i^{-1} c_4^3 A^2 \delta_i^2 \tau(\xi_i) + \ob(\delta_i),
\end{split}
\]
where we have used the fact that 
\[
\int_{\R^4} \frac{|y|^2-1}{(1+|y|^2)^4} = \frac{1}{3}\int \frac{1}{(1+|y|^2)^3}
\]
(check for instance \cite[Remark B.2]{MussoPistoiaIndiana2002}) and $|Z_i^0|\leq C|U_i|/\delta$. Moreover,
\[
\begin{split}
\lambda_i \mu_i^{-1} \int_\Omega (PU_i)Z_i^0 &= \lambda_i \mu_i^{-1} \int_\Omega (U_i-\delta_i A H(x,\xi_i) + \ob(\delta_i))(Z_i^0-AH(x,\xi_i)+\ob(1))\\
									&= \lambda_i \mu_i^{-1} \int_\Omega U_i Z_i^0 + \lambda_i \Ob(\delta_i)=\lambda_i \mu_i^{-1} c_4^2 \delta_i \int_{\frac{\Omega -\xi_i}{\delta_i}} \frac{|y|^2-1}{(1+|y|^2)^3} + \lambda_i \Ob(\delta_i)\\
									&=\lambda_i \mu_i^{-1} c_4^2 \delta_i \int_{\frac{\Omega -\xi_i}{\delta_i}} \frac{1}{(1+|y|^2)^2} + \lambda_i \Ob(\delta_i)=\lambda_i \mu_i^{-1}c_4^2 \omega_3 \delta_i |\ln \delta_i| + \lambda_i \Ob(\delta_i)
\end{split}
\]
and, for $j\neq i$,
\begin{multline*}
\beta \mu_i^{-1}\mu_j^{-1}\int_\Omega (PU_i)(PZ_i^0) (PU_j)^2 \\
		= \beta \mu_i^{-1}\mu_j^{-1} \int_\Omega (U_i-\delta_i AH(x,\xi_i)+\ob(\delta_i))(Z_i^0-AH(x,\xi_i)+\ob(1))(U_j^2+U_j \Ob(\delta_j)+\Ob(U_j^2))\\
		= \beta \mu_i^{-1}\mu_j^{-1}\int_\Omega U_i Z_i^0 U_j^2 + \beta \Ob(\delta_i \delta_j^2 \ln \delta_j)= \beta\Ob(\delta_i \delta_j^2 \ln \delta_j).
\end{multline*}

As for the remainder term, we can write
\begin{multline*}
\int_\Omega (f(\mu_i^{-1/2} PU_i)-f(\mu_i^{-1/2}PU_i + \phi_i))\mu_i^{1/2} (PZ_i^0)\\
=\int_\Omega (f'(\mu_i^{-1/2} PU_i)\phi_i+  f(\mu_i^{-1/2} PU_i)-f(\mu_i^{-1/2}PU_i + \phi_i))\mu_i^{1/2} (PZ_i^0) \\
+ \int_\Omega (f'(\mu_i^{-1/2}U_i)-f'(\mu_i^{-1/2}PU_i))\mu_i^{1/2}PZ_i^k
\end{multline*}
and, by reasoning exactly as in the proof of Lemma \ref{lemma:remainderestimate}, we deduce \[
|R^0_i(\deltab,\xib)| \leq C \sum_{i=1}^m\left( \lambda_i \delta_i^2+ \delta_i^3  \right).
\]

\medbreak

\noindent 2) For $k=1,\ldots, 4$,

\[
\begin{split}
\int_\Omega (f(\mu_i^{-1/2}U_i)-&f(\mu_i^{-1/2} PU_i + \phi_i)) \mu_i^{1/2} (PZ_i^k) = \int_\Omega \mu_i^{-1} (U_i^3-(PU_i)^3)PZ_i^k\\
								&= 3\mu_i^{-1}\int_\Omega \delta_i U_i^2 AH(x,\xi_i) PZ_i^k +\ob(\delta_i^2)= \mu_i^{-1} \int_\Omega  \delta_i \frac{\partial}{\partial_{x_k}}((PU_i)^3)A H(x,\xi_i) + \ob(\delta_i^2)  \\
								&=   - \mu_i^{-1}A\int_\Omega \frac{\delta_i^2}{2}  \frac{\partial \tau}{\partial \xi_i}(\xi_i) U_i^3 + \ob(\delta_i^2) = -\mu_i^{-1} A^2 \delta_i^3 \frac{c_4^3}{2}\frac{\partial \tau}{\partial \xi_i}(\xi_i) + \ob(\delta_i^2),
\end{split}
\]
\[
\begin{split}
\lambda_i \mu_i^{-1} \int_\Omega (PU_i) (PZ_i^k) &=\lambda_i \mu_i^{-1}\int_\Omega U_i Z_i^k + \lambda_i \Ob(\delta_i^2)\\
									&=\lambda_i \mu_i^{-1} \delta_i \int_{\R^4} \frac{y_k}{(1+|y|^2)^2}+\Ob(\delta_i^2)=\Ob(\delta_i^2)
\end{split}
\]
and, for $j\neq i$,
\[
\begin{split}
\beta \mu_i^{-1}\mu_j^{-1}\int_\Omega (PU_i)(PZ_i^0) (PU_j)^2 &= \beta \mu_i^{-1}\mu_j^{-1} \int_\Omega U_iZ_i^k U_j^2 + \beta \Ob(\delta_i^2 \delta_j^2 \ln \delta_j)\\
												&=\beta \Ob(\delta_i^2 \delta_j^2 \ln \delta_j),
\end{split}
\]
by using the fact that $\displaystyle \int_{\R^4} \frac{y_k}{(1+|y|^2)^3}=0$. Finally, reasoning as before, we deduce the estimate for the remainder term:
\[
|R^k_i(\deltab,\xib)| \leq C \sum_{i=1}^m\left( \lambda_i \delta_i^3+ \delta_i^4  \right). \qedhere
\]

\end{proof}

Now choose
\[
\delta_i=e^{-\frac{d_i}{\lambda_i}}, \qquad \text{ for some } d_i>0.
\]
The previous lemma implies that the expansion
\[
\widetilde E(\deltab,\xib)=\widetilde E(\mathbf{d},\xib)=\sum_{i=1}^m 16\mu_i^{-1} B + \Psi_{\mathbf{\lambda}}(\mathbf{d},\xib)+\sum_{i=1}^m \ob(e^{-\frac{2d_i}{\lambda_i}})
\]
is $C^1$ in $(\deltab,\xib)$, as $\eps:=\max\{\lambda_1,\ldots, \lambda_m\}\to 0^+$, uniformly in every compact set of $[0,+\infty)^m \times X$. Recall that, as before, we are defining $\sum_{i=1}^m \Psi_{\lambda_i}(d_i,\xi_i)=: \Psi_{\lambdab}(\mathbf{d},\xib)$, with

\[
\Psi_{\lambda_i}(d_i,\xi_i)= \sum_{i=1}^m e^{-\frac{2d_i}{\lambda_i}} \left( 8\sqrt{2} \mu_i^{-1}  A^2 \tau (\xi_i)-4\mu_i^{-1}\omega_3 d_i\right).
\]

We want to show that, as $\lambdab\to 0$, there exist $\mathbf{d}_{\lambdab},\xib_{\lambdab}$ such that
\[
\begin{cases}
e^{-\frac{2d_i}{\lambda_i}}(-4 \mu_i^{-1} \omega_3 -\frac{2}{\lambda_i} (8\sqrt{2} \mu_i^{-1} A^2 \tau(\xi_i)-4\mu_i^{-1}\omega_3 d_i) )+\ob(e^{-\frac{2d_i}{\lambda_i}})=0\\
e^{-\frac{2d_i}{\lambda_i}} 8\sqrt{2} \mu_i^{-1} A^2 \nabla \tau (\xi_i) + o(e^{-\frac{2 d_i}{\lambda_i}})=0.
\end{cases}
\]
or, equivalently,
\begin{equation}\label{eq:system_stable}
\begin{cases}
-4 \lambda_i \mu_i^{-1} \omega_3 -16 \sqrt{2} \mu_i^{-1} A^2 \tau (\xi_i) + 8 \mu_i^{-1} \omega_3 d_i +\ob(1)=0\\
\nabla \tau(x_i)+\ob(1)=0
\end{cases}
\end{equation}
as $\lambdab\to 0$, uniformly for $\deltab,\xib$ in a compact set of $[0,+\infty)^m \times X$.

\begin{proof}[Conclusion of the proof of Theorem \ref{thm:main2}]
Assume that the Robin function $\tau$ has $m$ distinct stable critical points: $\bar \xi_1,\ldots, \bar \xi_m$, in the sense of Definition \ref{def:stable}.
The conclusion follows as soon as we prove that, as $\lambdab\to 0$, there exists $(\mathbf{d}_{\lambdab},\xib_{\lambdab})$ solution of system \eqref{eq:system_stable}.

Define the map
\[
\Lambda(\mathbf{d},\xib)=\left(
\begin{array}{c} 
-16 \sqrt{2} \mu_1^{-1} A^2 \tau (\xi_1 ) + 8 \mu_1^{-1} \omega_3 d_1  \\
\vdots\\
-16 \sqrt{2} \mu_m^{-1} A^2 \tau (\xi_m ) + 8 \mu_m^{-1} \omega_3 d_m  \\
\nabla \tau(\xi_1)\\
\vdots\\
\nabla \tau(\xi_m)
\end{array} \right)
\]
which vanishes at 
\[
(d_1,\ldots, d_m,\xi_1,\ldots, \xi_m)=\left( \frac{2\sqrt{2} A^2}{\omega_3}\tau(\bar x_1),\ldots, \frac{2\sqrt{2} A^2}{\omega_3}\tau(\bar x_m), \bar \xi_1,\ldots, \xi_m  \right).
\]
Consider 
\[
U:= \prod_{i=1}^m U_i:= \prod_{i=1}^m \left[\frac{2\sqrt{2} A^2}{\omega_3}\tau(\bar x_i)-\delta,\frac{2\sqrt{2} A^2}{\omega_3}\tau(\bar x_i)+\delta\right]
\]
for $\delta$ small, and $V:=\Pi_{i=1}^m V_i$, where $V_i$ is a neighborhood of $\bar \xi_i$ considered in Definition \ref{def:stable}. Following \cite[Lemma 3.1]{MussoPistoiaIndiana2002}, we consider the following deformation from $\Lambda(\mathbf{d},\xib)$ to
\[
\widetilde \Lambda(\mathbf{d},\xib)=\left(
\begin{array}{c} 
-16 \sqrt{2} \mu_1^{-1} A^2 \tau (\bar \xi_1 ) + 8 \mu_1^{-1} \omega_3 d_1  \\
\vdots\\
-16 \sqrt{2} \mu_m^{-1} A^2 \tau (\bar \xi_m ) + 8 \mu_m^{-1} \omega_3 d_m  \\
\nabla \tau(\xi_1)\\
\vdots\\
\nabla \tau(\xi_m)
\end{array} \right)
\]
through the map  $H:[0,1]\times U\times V\to \R^{m+Nm}$ defined by 
\[
H(t,(\deltab,\xib))= t\Lambda(\mathbf{d},\xib)+(1-t) \widetilde \Lambda(\mathbf{d},\xib)
= \left( \begin{array}{c} -16 \sqrt{2} \mu_i^{-1} A^2 (t\tau(\xi_1)+(1-t) \tau(\bar \xi_1))+8 \mu_1^{-1} \omega_3 \lambda_1\\
				\vdots\\
				-16 \sqrt{2} \mu_m^{-1} A^2 (t\tau(\xi_m)+(1-t) \tau(\bar \xi_m))+8 \mu_m^{-1} \omega_3 \lambda_m \\
						\nabla \tau(\xi_1)\\
						\vdots\\
						\nabla \tau(\xi_m)
						 \end{array}\right)
			\]
We claim that $H(t,(\mathbf{d},\xib))\neq 0$ for every $(\mathbf{d},\xib)\in \partial (U\times V)$, $t\in [0,1]$. In fact, if $\mathbf{d}\in U$ and  $\mathbf{\xi}\in \partial V$, then $\nabla \tau(\xi_i)\neq 0$ for some $i$. If, instead, $\mathbf{d}\in \partial U$ and $\xi\in V$, then we have $d_i=\frac{2\sqrt{2} A^2}{\omega_3}\tau(\bar x_i)\pm \delta$ for some $i$, and  $-16 \sqrt{2} \mu_i^{-1} A^2 (t\tau(\xi_i)+(1-t) \tau(\bar \xi_i))+8 \mu_i^{-1} \omega_3 \lambda_i= 16\sqrt{2} \mu_i^{-1} A^2 t(\tau(\bar \xi_i)-\tau(\xi_i))\pm 8 \mu_i^{-1}\omega_3 d_i \delta$; hence either $\tau(\bar \xi_i)\neq 0$, or $\tau(\bar \xi_i)=0$ and $\tau(\xi_i)=\tau(\bar \xi_i)$. In every situation the claim follows, and by the invariance property of the degree, we conclude that
\[
\deg (\Lambda,U\times V,\mathbf{0})=\deg (\widetilde \Lambda,U\times V,\mathbf{0})\neq 0. \qedhere
\]
\end{proof}

\subsection{Proof of Theorem \ref{thm:main3}}

We will take any sequence $\beta=\beta(\lambdab)\to -\infty$ such that
\[
\beta=\ob\left(  	\exp\left( \frac{c_4}{2\omega_3 \lambda_i}A^2 \tau(\xi_i^0)\right)\right) \quad \forall i
\]
and work with $\delta_i=e^{-\frac{d_i}{\lambda_i}}$, with
\[
d_i \in \left[ \frac{c_4}{\omega_3 \lambda_i}A^2 \tau(\xi_i^0) -\gamma , \frac{c_4}{\omega_3 \lambda_i}A^2 \tau(\xi_i^0)+\gamma\right]
\]
for $\gamma<\frac{c_4}{4\omega_3 \lambda_i}A^2 \tau(\xi_i^0)$. In particular, $d_i\geq \frac{3c_4}{4\omega_3} A^2 \tau(\xi_i^0)$ and
\[
\delta_i \leq e^{- \frac{3c_4}{4\omega_3 \lambda_i} A^2 \tau (\xi_i^0)}.
\]
Denote, to simplify notations, $C_i:= \frac{c_4 A^2}{\omega_3}\tau(\xi_i^0)$, so that
\[
\beta=\ob(e^\frac{C_i}{2\lambda_i}) \ \forall i \qquad \text{ and } \qquad \delta_i=\delta_i(d_i,\xi_i) \leq e^{-\frac{3C_i}{4\lambda_i}}
\]
as $\lambdab\to 0$ (we stress the fact that $\deltab$ is a (explicit) function  of $\mathbf{d}$ and $\xib$).

We will bliefly explain here why, in order to prove Theorem \ref{thm:main3}, we can repeat, for these choices of $\beta$ and $\deltab$, the proofs of Theorems \ref{thm:main1} and \ref{thm:main2}:
\begin{enumerate}
\item[a)] Lemma \ref{lemma:L_invertible} continues to hold in this context, as soon as we realize that
\[
|\beta|\delta_i^2 =\ob(1) e^\frac{2C_i}{2\lambda_i} e^{-\frac{3C_i}{2\lambda_i}}\to 0 \qquad \text{ as } \lambda_i\to 0,
\]
so that also $|\beta|\delta_i\delta_j\to 0$, $|\beta|\delta_i^2 \delta_j^2\to 0$ and $|\beta|\delta_i^2 \sqrt{|\ln \delta_i|}\to 0$. Thus, even if $\beta$ varies, the estimates  \eqref{w_{in}=2nd}, \eqref{eq:B_n}, \eqref{eq:C_n}, \eqref{eq:B'_n1} and \eqref{eq:C'_n} --the only places where $\beta$ appears --  continue to hold.
\item[b)] As for Proposition \ref{prop:estimate_error}, we obtain, for $c$ independent from $\beta$,
\[
\|\phib\|_{H^1_0}\leq c\sum_{i=1}^m (\lambda_i \delta_i +\delta_i^2 |\beta| \delta_i^2).
\]
In fact, recalling \eqref{eq:R_deltaxi}, we have (using the notation of the proof)
\[
\left \|\widetilde R_{\deltab,\xib} \right\|_{4/3} \leq c\sum_{i=1}^m (\delta_i^2+\lambda_i \delta_i+ |\beta| \delta_i^2),
\]
so that
\[
\left\| T_{\deltab,\xib}(\phib) \right\|_{H^1_0}\leq c'\left( (1+|\beta|)(\|\phib\|_{H^1_0}^2 +\|\phib\|_{H^1_0}^3)+\sum_{i}(\delta_i^2+\lambda_i \delta_i +\beta\delta_i^2) \right)
\]
and for some $c''$:
\begin{multline*}
T_{\deltab,\xib}\left(\left\{ \phib\in K_{\deltab,\xib}^\perp:\ \|\phib\|_{H^1_0}\leq \sum_{i=1}^m c'' \left(\delta_i^2+\lambda_i\delta_i + |\beta|\delta_i^2\right)  \right\}\right) 		\\   \subseteq \left\{ \phib\in K_{\deltab,\xib}^\perp:\ \|\phib\|_{H^1_0}\leq c'' \sum_{i=1}^m \left(\delta_i^2+\lambda_i\delta_i+|\beta|\delta_i^2\right)  \right\}
\end{multline*}
simply because $|\beta|^3\delta_i^4\leq |\beta|\delta_i^2$ is equivalent to $|\beta|^2\delta_i^2\leq 1$, which holds for small $\lambdab$ since
\[
|\beta|^2\delta_i^2 =\ob(1) e^\frac{C_i}{\lambda_i}e^{-\frac{3C_i}{2\lambda_i}}\to 0 \quad \text{ as } \lambda_i\to 0.
\]
\item[c)] Having this estimate at hand for the remainder term $\phib$, we can prove, exactly as in Lemma \ref{lemma:differentiability}, that the map $(\mathbf{d},\xib)\mapsto \phi^{\deltab,\xib}$ is $C^1$: in fact, in estimate \eqref{eq:comparing2}, we have
\[
\begin{split}
|\beta| \|\phib^{\deltab,\xib}\| &\leq c \sum_{i=1}^m (\lambda_i |\beta|\delta_i + |\beta| \lambda_i \delta_i +|\beta|\delta_i^2)\\
					& \leq \tilde c \sum_{i=1}^m (\lambda_i e^\frac{C_i}{2\lambda_i}e^{-\frac{3C_i}{4\lambda_i}} + e^\frac{C_i}{2\lambda_i} e^{-\frac{3C_i}{2\lambda_i}})\to 0	
	\end{split}		
\]
as $\lambda_i\to 0$.
\item[d)] Finally, as for the energy estimate, since
\[
|\beta|\delta_i^2 \ln |\delta_i| \to 0 \qquad \text{ as } \lambda_i\to 0,
\]
the we can repeat the word by word the arguments of Section \ref{sec:Energy_expansion} and Lemma \ref{lemma:C^1energyexpansion}, and obtain the estimate \eqref{eq:expansionfinal}. Once this is known, the rest of the proof follows exactly as in the previous two subsection.\end{enumerate}


\appendix

\section{Auxiliary estimates}\label{sec:Appendix}

We start by recalling here the notations
\[
U_{\delta,\xi}(x)=c_4\frac{\delta}{\delta^2+|x-\xi|^2},\qquad \psi^0_{\delta,\xi}=\delta \frac{\partial U_{\delta,\xi}}{\partial \delta} \qquad \text{and} \qquad \psi^j_{\delta,\xi}=\delta\frac{\partial U_{\delta,\xi}}{\partial \xi_j},\ j=1,\ldots, 4.
\]

\begin{lemma}\label{lemma:expansion_of_PU}
Given any $K\Subset \Omega\subset \R^4$, we have
\begin{equation}\label{eq:expansion1}
PU_{\delta,\xi}=U_{\delta,\xi}-\delta A H(\cdot,\xi)+R_{\delta,\xi},
\end{equation}
with
\[
\|R_{\delta,\xi}\|_\infty=\textrm{o}(\delta),\quad \left\|\frac{\partial R_{\delta,\xi}}{\partial \delta} \right\|_\infty=\textrm{o}(1),\quad \left\|\frac{\partial R_{\delta,\xi}}{\partial \xi_j}\right\|_\infty=\textrm{o}(\delta)
\]
as $\delta\to 0$, uniformly for $\xi\in K$, with
\[
A:=\int_{\R^4} U_{1,0}^3=c_4^3\int_{\R^4}\frac{1}{(1+|y|^2)^3}=\frac{c_4}{\alpha_4}.
\]
In particular, we have
\begin{equation}\label{eq:expansion2}
PU_{\delta,\xi}=U_{\delta,\xi}-\delta A H(\cdot,\xi)+\textrm{o}(\delta),\qquad P\psi^0_{\delta,\xi}=\psi_{\delta,\xi}^0-\delta A H(\cdot,\xi)+\ob(\delta)
\end{equation}
and
\begin{equation}\label{eq:expansion3}
P\psi_{\delta,\xi}^j=\psi_{\delta,\xi}^j-\delta^2 A\frac{\partial H}{\partial \xi}(\cdot,\xi)+\ob(\delta^2), \qquad j=1,\ldots, 4,
\end{equation}
as $\delta\to 0$, uniformly for $\xi\in K$.
\end{lemma}
\begin{proof}
See Proposition 1 in \cite{Rey} \qedhere
\end{proof}

\begin{lemma}\label{lemma:integral_estimates} Take $K\Subset \Omega\subset \R^4$.
Then for every $0<p< 2$,
\[
\int_\Omega U_{\delta,\xi}^p=\textrm{O}(\delta^p) \qquad \text{ as } \delta\to 0,
\]
while
\[
\int_\Omega U_{\delta,\xi}^2= c_4^2\delta^2 \omega_3 |\ln \delta|+\Ob(\delta^2) \qquad \text{ as } \delta\to 0
\]
and, for $2<p<4$,
\[
\int_\Omega U_{\delta,\xi}^p= \textrm{O}(\delta^{4-p}) \qquad \text{ as } \delta\to 0,
\]
uniformly for $\xi\in K$, where $\omega_3$ denotes the measure of the unit sphere $S^3\subset \R^4$.
\end{lemma}
\begin{proof}
Take $R=\dist(\xi,\partial \Omega)/2$. Then
\[
\begin{split}
\int_\Omega U_{\delta,\xi}^p& =c_4^p \int_{B_{R}(\xi)} \frac{\delta^p}{(\delta^2+|x-\xi|^2)^p}+\Ob(\delta^p)\\
		&=c_4^p \delta^{4-p} \omega_3 \int_0^{R/\delta} \frac{t^3}{(1+t^2)^p}\, dt+\Ob(\delta^p).
\end{split}
\]
We have $t^3/(1+t^2)^p\leq t^{3-2p}$ for every $t\geq 0$. Thus, if $p\neq 3/2$,
\[
\int_\Omega U_{\delta,\xi}^p=\Ob(\delta^{4-p}) + c_4^p \delta^{4-p} \omega_3 \left[ t^{4-2p}\right]_{t=1}^{t=R/\delta}+\Ob(\delta^p)=\Ob(\delta^{4-p})+\Ob(\delta^p).
\]
Thus $\int_\Omega U_{\delta,\xi}^p=\Ob(\delta^p)$ if $0<p<2$, and $\int_\Omega U_{\delta,\xi}^p=\Ob(\delta^{4-p})$ if $2<p<4$.

\medbreak

Likewise, we have that
\[
\begin{split}
\int_\Omega U_{\delta,\xi}^2 &= c_4^2 \delta^2 \omega_3 \int_0^{R/\delta} \frac{t^3}{(1+t^2)^2}\, dt + \Ob(\delta^2)\\
					    &=\frac{c_4^2}{2}\delta^2 \omega_3 \left[\frac{1}{1+t^2}+\ln (1+t^2)\right]^{t=R/\delta}_{t=0}+\Ob(\delta^2)\\
					     &=-c_4^2 \delta^2 \omega_3 \ln \delta+\Ob(\delta^2),
\end{split}
\]
as $\delta\to 0$. \qedhere
\end{proof}

Next we will present some asymptotic estimates related to the competition term of the system under consideration. For that, we will need the following simple pointwise estimates.

\begin{lemma}\label{lemma:auxiliaryestimate}
Let $F:\R\to \R$ be defined by $F(s)=(s^+)^4/4$. Then there exists $c>0$ such that, for every $a,b\in \R$,
\[
|F(a+b)-F(a)-F'(a)b|\leq c (a^2b^2 +b^4),
\]

\[
|F'(a+b)-F'(a)-F''(a)b|\leq c(|a| b^2 +|b|^3),
\]
and
\[
|F''(a+b)-F''(a)|\leq c (|a| |b| + b^2).
\]
Moreover, for every $p>1$ there exists $C>0$ such that
\[
\left| |a+b|^p-|a|^p\right|\leq C \left( |a|^{p-1}|b| + |b|^p \right),
\]
for every $a,b\in \R$.
\begin{proof}
The proof is very elementary, and follows simply from a Taylor's expansion with Lagrange-type remainder.
\end{proof}
\end{lemma}

\begin{lemma}\label{lemma:upper_estimate_auxiliar}
 Given $p,q>0$ and $\eta>0$ small, we have that
 \[
 \int_\Omega U_{\delta_1,\xi_1}^p U_{\delta_2,\xi_2}^q \leq \Ob(\delta_2^q)\int_\Omega U_{\delta_1,\xi_1}^p + \Ob(\delta_1^p) \int_\Omega U_{\delta_2,\xi_2}^q+\Ob(\delta_1^p\delta_2^q).
 \]
 as $(\delta_1,\delta_2)\to (0,0)$, uniformly for all $\xi_1,\xi_2\in \Omega$ such that $|\xi_1-\xi_2|\geq 2\eta$, $\dist(\xi_1,\partial \Omega),\dist(\xi_2,\partial \Omega)\geq 2\eta$.
\end{lemma}
\begin{proof}
This is a simple consequence of the fact that
\[
 \int_\Omega U_{\delta_1,\xi_1}^p U_{\delta_2,\xi_2}^q = \int_{B_\eta(\xi_1)} U_{\delta_1,\xi_1}^p U_{\delta_2,\xi_2}^q+ \int_{B_\eta(\xi_2)} U_{\delta_1,\xi_1}^p U_{\delta_2,\xi_2}^q +\int_{\Omega \setminus (B_\eta(\xi_1)\cup B_{\eta}(\xi_2))}U_{\delta_1,\xi_1}^p U_{\delta_2,\xi_2}^q
\]
together with the fact that $U_{\delta_i,\xi_i}\leq C \delta_i$ on $\Omega \setminus B_\eta(\xi_i)$, for some $C>0$ independent of $\delta_i$.
\end{proof}

\begin{lemma}\label{lemma:auxiliary_lemmas_appendix}
 Given $j=0,\ldots, 4$ and $\eta>0$ small, we have
\[
\left\|(PU_{\delta_2,\xi_2})^2(P\psi_{\delta_1,\xi_1}^j) \right \|_{4/3}=\Ob(\delta_{1}\delta_2) \quad \text{ and }\quad  \left\|(PU_{\delta_1,\xi_1})(PU_{\delta_2,\xi_2})(P\psi_{\delta_1,\xi_1}^j )\right\|_{4/3}=\Ob(\delta_{1}\delta_2)
\]
as $(\delta_1,\delta_2)\to (0,0)$, uniformly for all $\xi_1,\xi_2\in \Omega$ such that $|\xi_1-\xi_2|\geq 2\eta$, $\dist(\xi_1,\partial \Omega),\dist(\xi_2,\partial \Omega)\geq 2\eta$.
\end{lemma}
\begin{proof}
To simplify notations, define $U_i:=U_{\delta_i,\xi_i}$ and $\psi_i^j:=\psi^j_{\delta_i,\xi_i}$, for $i=1,2$, $j=0,\ldots, 4$. We have ($j=0$)
\[
\left \|(PU_2)^2(P\psi_1^0) \right\|_{4/3}^{4/3}=\int_\Omega |U_{2}|^{8/3}|\psi^0_1|^{4/3}+R_{\deltab,\xib},
\]
with
\begin{multline*}
R_{\deltab,\xib}=\int_\Omega \left[ (PU_2)^{8/3}-U_2^{8/3}\right] \left[ |P\psi_1^0|^{4/3}-|\psi_1^0|^{4/3} \right]+\int_\Omega U_2^{8/3}\left[ |P\psi_1^0|^{4/3}-|\psi_1^0|^{4/3} \right]\\
+\int_\Omega |\psi_1^0|^{4/3}\left[ (PU_2)^{8/3}-U_2^{8/3}\right] , 
\end{multline*}
so that, by taking in consideration Lemmas \ref{lemma:expansion_of_PU}, \ref{lemma:auxiliaryestimate} and \ref{lemma:upper_estimate_auxiliar}, and the fact that $|\psi_1^0|\leq C U_1$ for some $C>0$, we have
\begin{multline*}
|R_{\deltab,\xib}| \leq \int_\Omega \left| U_2^{5/3}\Ob(\delta_2) +\Ob(\delta_2^{8/3}) \right| \left| |\psi_1^0|^{1/3}\Ob(\delta_1)+\Ob(\delta_1^{4/3})\right|+\int_\Omega U_2^{8/3} \left| |\psi_1^0|^{1/3}\Ob(\delta_1) + \Ob(\delta_1^{4/3})\right| \\
+ \int_\Omega |\psi_1^0|^{4/3}\left| U_2^{5/3}\Ob(\delta_2)+\Ob(\delta_2^{8/3})\right| =\Ob(\delta_1^{4/3}\delta_2^{4/3})
\end{multline*}
Since also 
\[
\begin{split}
\int_\Omega (U_2)^{8/3}|\psi_1^0|^{4/3} &\leq  C \int_\Omega (U_2)^{8/3}(U_1)^{4/3}\leq \Ob(\delta_1^{4/3})\int_\Omega (U_2)^{8/3}+\Ob(\delta_2^{8/3})\int_\Omega (U_1)^{4/3}+\Ob(\delta_1^{4/3}\delta_2^{4/3})\\
					&=\Ob(\delta_1^{4/3}\delta_2^{4/3}),
\end{split}
\]
we have
\[
\left\| (PU_2)^2(P\psi_1^0)\right\|_{4/3}^{4/3}=\Ob(\delta_1^{4/3}\delta_2^{4/3}).
\]

For $j\geq 1$, we can reason in an analogous way: by using the fact that $|\psi_1^j|\leq C U_1^2$ for some $C>0$, we obtain
\[
\|(PU_2)^2(P\psi_1^j)\|_{4/3}^{4/3}=\int_\Omega (U_{2})^{8/3}|\psi^j_1|^{4/3}+\Ob(\delta_1^{4/3}\delta_2^{4/3})=\Ob(\delta_1^{4/3}\delta_2^{4/3}).
\]

\medbreak

As for the second conclusion of the lemma, reasoning in the same line, we write ($j=0$)
\[
\|(PU_1)(PU_2)(P\psi_1^0)\|_{4/3}^{4/3}=\int_\Omega (U_{1})^{4/3} (U_2)^{4/3}|\psi^0_1|^{4/3}+\tilde R_{\deltab,\xib},
\]
with
\begin{multline*}
\tilde R_{\deltab,\xib}= \int_\Omega \left[ (PU_1)^{4/3}-U_1^{4/3}\right] \left[ (PU_2)^{4/3}-U_2^{4/3}\right] \left[ |P\psi_1^0|^{4/3}-|\psi_1^0|^{4/3}\right] \\
	+ \int_\Omega U_1^{4/3} U_2^{4/3} \left[ |P\psi_1^0|^{4/3}-|\psi_1^0|^{4/3}\right] + \int_\Omega U_1^{4/3} \left[ (PU_2)^{4/3}-U_2^{4/3}\right] |\psi_1^0|^{4/3}\\
	+\int_\Omega \left [ (PU_1)^{4/3}-U_1^{4/3}\right] U_2^{4/3} |\psi_1^0|^{4/3}+\int_\Omega \left[ (PU_1)^{4/3}-U_1^{4/3}\right] \left[ (PU_2)^{4/3}-U_2^{4/3}\right] |\psi_1^0|^{4/3}\\
	+  \int_\Omega \left[ (PU_1)^{4/3}-U_1^{4/3}\right] U_2^{4/3} \left[ |P\psi_1^0|^{4/3}-|\psi_1^0|^{4/3}\right] \\
	+ \int_\Omega U_1^{4/3} \left[ (PU_2)^{4/3}-U_2^{4/3}\right] \left[ |P\psi_1^0|^{4/3}-|\psi_1^0|^{4/3}\right] 
\end{multline*}
By using once again Lemmas \ref{lemma:expansion_of_PU}, \ref{lemma:auxiliaryestimate} and \ref{lemma:upper_estimate_auxiliar}, we can prove that $\tilde R_{\deltab,\xib}=\Ob(\delta_1^{4/3}\delta_2^{4/3})$. Since moreover (by Lemma \ref{lemma:upper_estimate_auxiliar})
\[
\int_\Omega U_{1}^{4/3} U_2^{4/3}|\psi^0_1|^{4/3}\leq C\int_\Omega U_1^{8/3}U_2^{4/3}=\Ob(\delta_1^{4/3}\delta_2^{4/3}), 
\]
we conclude, as wanted, that
\[
\|(PU_1)(PU_2)(P\psi_1^0)\|_{4/3}^{4/3}=\Ob(\delta_1^{4/3}\delta_2^{4/3}).
\]
The fact that
\[
\|(PU_1)(PU_2)(P\psi_1^j)\|_{4/3}^{4/3}=\Ob(\delta_1^{4/3}\delta_2^{4/3}) \qquad \text{ for } j=1,\ldots, 4,
\]
follows in an analogous way, using this time the estimate: $|\psi_1^j|\leq C U_1^{2}$, for some $C>0$.
\end{proof}

\section*{Acknowledgments}

Angela Pistoia was supported by GNAMPA and Sapienza Fondi di Ricerca.

Hugo Tavares was partially supported by Funda\c c\~ao para a Ci\^encia e Tecnologia
through the program Investigador FCT and the project PEst-OE/EEI/-LA0009/2013, as well as by the ERC Advanced Grant 2013 n.339958 ``Complex Patterns for Strongly Interacting Dynamical Systems - COMPAT''.

\end{document}